\documentclass[journal]{IEEEtran}

\usepackage{chngcntr}
\usepackage{bibunits}
\usepackage{array}
\usepackage{textcomp}
\usepackage{stfloats}
\usepackage{url}
\usepackage{verbatim}
\usepackage{graphicx}
\usepackage{enumitem}
\usepackage{cite}
\usepackage{booktabs}
\usepackage{subcaption} 
\usepackage{xcolor}
\usepackage{adjustbox}
\usepackage{tabularx}
\usepackage{booktabs}
\usepackage{multirow}
\usepackage{arydshln}
\usepackage{hyperref}
\renewcommand{\arraystretch}{1.3}

\usepackage{amsmath,amssymb,amsthm}

\graphicspath{{./pictures/}}

\usepackage{algorithmicx,algorithm}
\usepackage{algpseudocode}  
\newtheorem{theorem}{Theorem}

\newtheorem{remark}{Remark}

\newtheorem{proposition}{Proposition}
\newtheorem{corollary}{Corollary}

\hypersetup{
colorlinks=true,
linkcolor=black,
urlcolor=black
}

\newcommand{\cM}{{\cal M}}
\newcommand{\cC}{{\cal C}}

\newcommand{\cT}{{\cal T}}
\newcommand{\hy}{{\hat y}}
\newcommand{\hz}{{\hat z}}
\newcommand{\hx}{{\hat x}}
\newcommand{\hw}{{\hat w}}
\newcommand{\by}{{\bar y}}
\newcommand{\bz}{{\bar z}}
\newcommand{\bx}{{\bar x}}
\newcommand{\bw}{{\bar w}}
\newcommand{\bu}{{\bar u}}
\newcommand\hcT{{\widehat { \cal T}}}
\newcommand\hcF{{\widehat { \cal F}}}
\newcommand\tcT{\widetilde{\mathcal{T}}}
\newcommand\tcF{\widetilde{\mathcal{F}}}

\ifCLASSINFOpdf
\else
\fi

\defaultbibliography{HOT_ref} 
\defaultbibliographystyle{IEEEtran} 

\begin{document}
\title{HOT: An Efficient Halpern Accelerating Algorithm for Optimal Transport Problems}

\author{Guojun Zhang,~ Zhexuan Gu, Yancheng Yuan, Defeng Sun
\thanks{Guojun Zhang and Zhexuan Gu contribute equally to this manuscript.}
\thanks{Corresponding author: Yancheng Yuan.}
\thanks{The authors are with The Hong Kong Polytechnic University. E-mail: guojun.zhang@connect.polyu.hk, zhexuan.gu@connect.polyu.hk, yancheng.yuan@polyu.edu.hk, defeng.sun@polyu.edu.hk.}}

\markboth{}%
{}

\maketitle

\begin{bibunit}

\begin{figure*}[h]
    \centering
    \includegraphics[width=0.6\textwidth,]{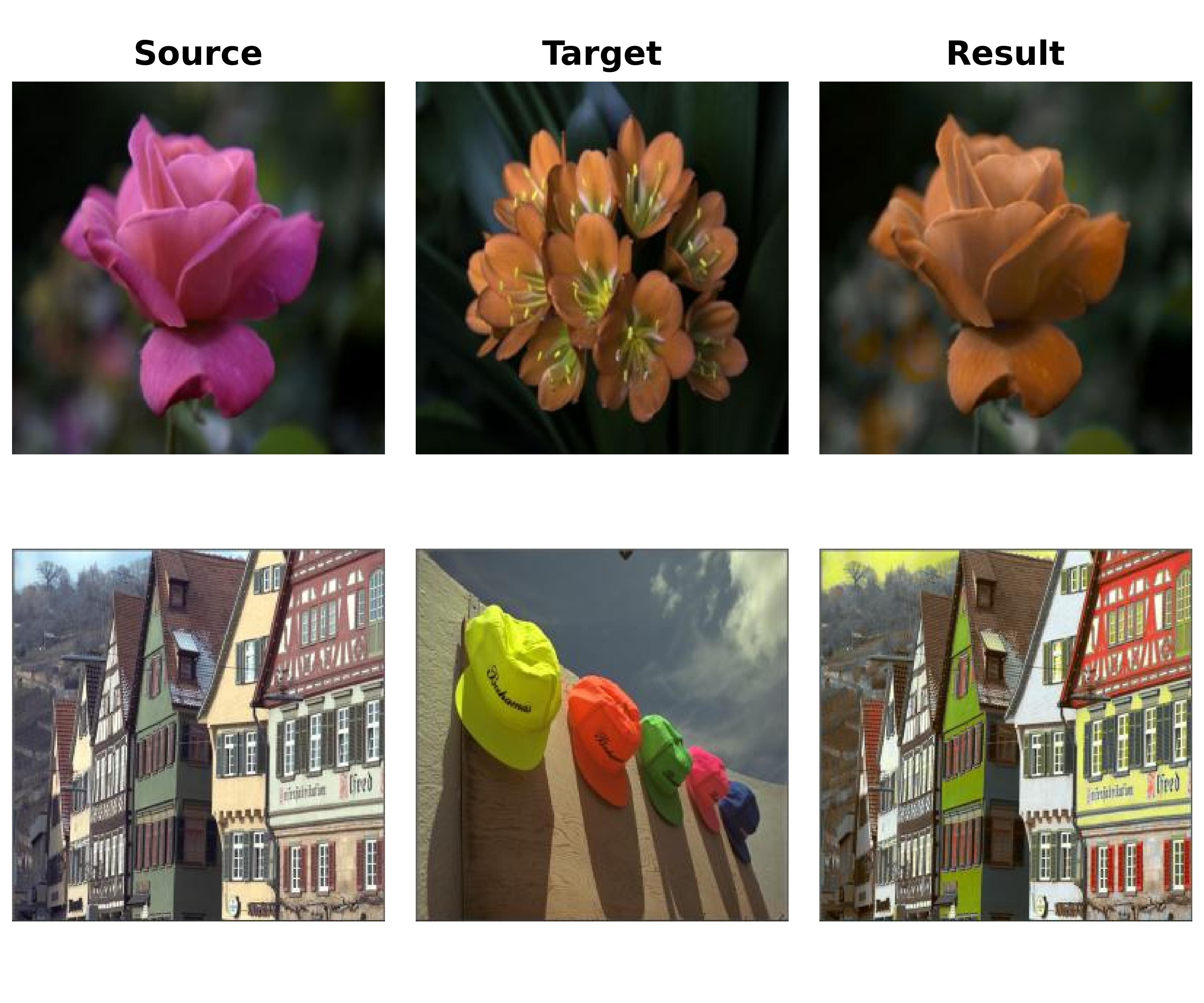}
    \caption{Selected examples of color transfer based on the reduced optimal transport model with the optimal transport plan recovered by Algorithm \ref{alg:transportplan}.}
    \label{fig:color_transfer_head}
\end{figure*}

\begin{abstract}
This paper proposes an efficient HOT algorithm for solving the optimal transport (OT) problems with finite supports. We particularly focus on an efficient implementation of the HOT algorithm for the case where the supports are in $\mathbb{R}^2$ with ground distances calculated by $L_2^2$-norm. Specifically, we design a Halpern accelerating algorithm to solve the equivalent reduced model of the discrete OT problem. Moreover, we derive a novel procedure to solve the involved linear systems in the HOT algorithm in linear time complexity. Consequently, we can obtain an $\varepsilon$-approximate solution to the optimal transport problem with $M$ supports in $O(M^{1.5}/\varepsilon)$ flops, which significantly improves the best-known computational complexity. We further propose an efficient procedure to recover an optimal transport plan for the original OT problem based on a solution to the reduced model, thereby overcoming the limitations of the reduced OT model in applications that require the transport plan. We implement the HOT algorithm in PyTorch and extensive numerical results show the superior performance of the HOT algorithm compared to existing state-of-the-art algorithms for solving the OT problems.   
\end{abstract}


\IEEEpeerreviewmaketitle

\section{Introduction}

\IEEEPARstart{T}hE Kantorovich-Wasserstein (KW) distance has become a prime choice for measuring the similarity between two probability distributions. It has demonstrated remarkable success in various applications, including color transfer \cite{pitie2007linear,bonneel2013example,solomon2015convolutional}, texture synthesis and mixing \cite{dominitz2009texture}, registration and warping \cite{haker2004optimal}, transport-based morphometry \cite{basu2014detecting}, and hypothesis testing \cite{del1999tests}, among others. Despite its powerful geometric framework for comparing probabilities, the KW distance is computationally expensive in general \cite{villani2003topics, peyre2019computational}.  
Specifically, it requires solving an optimal transport (OT) problem, which is a (large-scale) linear programming (LP) in a discrete setting. 
Standard methods, such as the simplex method and the interior point method, suffer from high computational complexity relative to the problem size. Furthermore, these methods are difficult to parallelize, which can hardly benefit from the modern powerful graphics processing units (GPUs). Consequently, solving the LP problem of the discrete OT problem remains daunting in modern data-driven applications due to high computational and memory costs. This paper addresses these two challenges by proposing an efficient and easily parallelizable algorithm for solving an equivalent reduced model of the OT problem.
\subsection{Related work and existing challenges}
When it comes to computing the KW distance between two discrete probability distributions with $M$ supports, there are mainly two popular approaches: (i) Computing an approximated KW distance by solving the optimal transport problem with an additional entropy regularization \cite{cuturi2013sinkhorn}; (ii) Computing the KW distance via solving the corresponding LP problem \cite{orlin1988faster}. The readers can refer to \cite{peyre2019computational} and the references therein for a more detailed discussion of the algorithms for solving OT problems. Before introducing our new algorithm, we briefly discuss the challenges of the aforementioned approaches.

{\textbf{Challenges with the Entropy-regularized approach:}} Due to its scalability, the {Sinkhorn} algorithm and its improved versions have been widely adopted to compute an approximation of the KW distance in applications \cite{cuturi2013sinkhorn,lin2019efficient,dvurechensky2018computational,guminov2021combination}. In particular, the {Sinkhorn} algorithm can efficiently solve the regularized OT problem when the regularization parameter is moderate (i.e., no less than $10^{-2}$). However, a high-quality approximation of the KW distance is important for better performance in many applications, which requires solving the regularized OT problem with a small regularization parameter. Unfortunately, a small regularization parameter will usually cause numerical issues and
a slower convergence for the {Sinkhorn} algorithm.  Some stabilized and rescaling techniques \cite{schmitzer2019stabilized} have been proposed to improve the robustness of solving the regularized OT problems, but the efficiency of the stabilized algorithms is unsatisfactory compared to the Sinkhorn algorithm.

{\textbf{Challenges with the LP approach:}} 
Along this line, the interior point method \cite{pele2009fast} and the network simplex method \cite{goldberg1989network,gabow1991faster} are popular choices for obtaining solutions with high accuracy to the moderate scale LP problem. Recently, a {semismooth Newton based inexact proximal augmented Lagrangian} method \cite{li2020asymptotically} has been proposed for solving linear programming problems which can solve the OT problem as a special case. The semismooth Newton based algorithm can exploit the sparsity of the solution by the generalized Jacobian and show superior numerical performance compared to Gurobi in some examples.  However, these solvers are not applicable for solving the problem on a very large scale due to the high computational complexity. The urgent need to solve large-scale OT problems in applications inspires extensive research in designing first-order methods, such as the {Douglas-Rachford splitting algorithm \cite{mai2022fast}} and primal-dual hybrid gradient method (PDHG) \cite{zhu2008efficient,esser2010general,chambolle2011first,applegate2021practical}. Jambulapati et al. \cite{jambulapati2019direct} proposed an algorithm based on a dual-extrapolation algorithm to achieve an $\widetilde{O}({M}^2/\varepsilon)$ computational complexity bound, where $M$ is the number of supports, for obtaining an $\varepsilon$-approximate solution (in terms of objective function value)\footnote{Despite its better complexity bound, the empirical performance of this algorithm is not as efficient as other algorithms, such as accelerated gradient method \cite{dvurechensky2018computational,guminov2021combination}.}. Recently, Zhang et al. \cite{zhang2022efficient} proposed an efficient Halpern-Peaceman-Rachford (HPR) algorithm for solving the OT model and the Wasserstein barycenter problem, which can obtain an $\varepsilon$-approximate solution (in terms of the Karush-Kuhn-Tucker (KKT) residual) of the OT model in $O({M}^2/\varepsilon)$ flops.
We summarize some known complexity results for solving the OT problem in Table \ref{tab:complexity}. Readers can refer to {\cite{tupitsa2022numerical, zhang2022efficient,khamis2024scalable}} and the references therein for a more detailed discussion.


\begin{table}[ht]
  
\center
\caption{Selected known complexity results for solving OT problem (\small {$C$ represents the largest elements of the cost matrix, while $R$ denotes the distance between the initial point and the solution set.})}
\label{tab:complexity}
\renewcommand{\arraystretch}{1.5}
\resizebox{1\linewidth}{!}{%
\begin{tabular}{ccc}
\hline
Algorithm                                       & Time complexity result                    & Space complexity result \\ \hline
Sinkhorn \cite{dvurechensky2018computational}                 & $\widetilde{O}({M^2C^2}/{\varepsilon^{2}})$       & ${O}({M^2})$     \\
APDAGD \cite{dvurechensky2018computational,lin2022efficiency} & $\widetilde{O}(M^{2.5}C/\varepsilon)$             & ${O}({M^2})$     \\
Greenkhorn \cite{lin2022efficiency}             & $\widetilde{O}(M^2C^2/\varepsilon^{2})$   & ${O}({M^2})$            \\
Accelerated Sinkhorn \cite{lin2022efficiency}                 & $\widetilde{O}(M^{7/3}C^{4/3}/\varepsilon^{4/3})$ & ${O}({M^2})$     \\
AAM \cite{guminov2021combination}               & $\widetilde{O}(M^{2.5}C/\varepsilon)$     & ${O}({M^2})$            \\
Dual extrapolation \cite{jambulapati2019direct} & $\widetilde{O}(M^2C/\varepsilon)$         & ${O}({M^2})$            \\
HPD \cite{chambolle2022accelerated}             & $\widetilde{O}({M^{2.5}C}/{\varepsilon})$ & ${O}({M^2})$            \\
HPR \cite{zhang2022efficient}\footnote{Zhang, G., Yuan, Y., \& Sun, D. (2022). An Efficient HPR Algorithm for the Wasserstein Barycenter Problem with $ O ({\rm Dim (P)}/\varepsilon) $ Computational Complexity. arXiv preprint arXiv:2211.14881. (Under review at JMLR)}                   & ${O(M^{2}R}/{\varepsilon})$               & ${O}({M^2})$            \\
\textbf{HOT (Ours)}                             & $\mathbf{O(M^{1.5}R/\varepsilon)}$  & $\mathbf{O(M^{1.5})}$ \\ \hline
\end{tabular}
}
\end{table}

Beyond the challenges above in computational efficiency, all these algorithms for the original OT problem with ${M}$ supports require at least a memory cost of $O({M}^2)$. This memory cost makes it forbidden to compute the KW distance of two distributions with massive supports (i.e., the OT problem for computing the KW distance of two $512\times512$ grey images has more than $6.8\times 10^{10}$ variables). To address this issue, researchers have proposed some approximate models for the OT problem, such as the linear OT framework \cite{wang2013linear}, the sliced Wasserstein distance \cite{bonneel2011displacement}, and the approximated earth mover's distance \cite{indyk2003fast, shirdhonkar2008approximate,leeb2016holder}. While these approaches alleviate memory issues and simplify computations, they inevitably introduce a loss of exactness. This limitation has motivated further research into equivalent reduced models for the OT problem. When the ground distances between supports in $\mathbb{R}^{d}$ are calculated by $L_1$-norm, Ling and Okada proposed an equivalent reduced model with $O(dM)$ memory cost to calculate the earth mover distance (equivalent to the KW distance) \cite{ling2007efficient}. Recently, Auricchio et al. \cite{auricchio2018computing} extended the idea to the case where the ground distances between supports are calculated by $L_2^2$-norm  and proposed an equivalent reduced model with $O(dM^{\frac{d+1}{d}})$ memory cost. The authors in \cite{auricchio2018computing} adopted the Network Simplex method to solve the reduced model and demonstrated superior performance in terms of computational and memory efficiency compared to the {Sinkhorn} and the convolutional {Sinkhorn} method \cite{solomon2015convolutional,solomon2018optimal} on examples of moderate scale. Unfortunately, the efficiency becomes unsatisfactory for very large-scale problems (see Section \ref{sec: Experiment} for details). Moreover, the transport plan is not available if we solve the reduced model, which is critical for a wide class of applications, such as color transfer \cite{pitie2007linear,ferradans2014regularized}, texture synthesis \cite{dominitz2009texture}, and domain adaptation \cite{courty2016optimal}.

\subsection{Contributions}
Motivated by the recent advancements in the accelerated algorithms based on Halpern iteration \cite{sun2024accelerating,zhang2022efficient,kim2021accelerated,tran2021halpern,lieder2021convergence,yang2025accelerated}, we propose an efficient \textbf{H}alpern accelerating method for solving the reduced \textbf{OT} problem, which is abbreviated as ``HOT'' for convenience, to address the challenges in computing the KW distance with a finite number of supports. We particularly focus on an efficient implementation of the HOT algorithm for the case where the supports are in $\mathbb{R}^2$ with ground distances calculated by $L_2^2$-norm, which includes a wide class of applications as aforementioned. Specifically, HOT adopts a first-order algorithm with Halpern acceleration to solve the equivalent reduced OT model, which can obtain an $\varepsilon$-approximate solution in $O(1/\varepsilon)$ iterations \cite{sun2024accelerating,zhang2022efficient}. More importantly, we design a fast procedure for solving the subproblems with linear time complexity by fully exploiting the problem structure. This also makes the popular alternating direction method of multipliers (ADMM) \cite{glowinski1975approximation,gabay1976dual} scalable for solving the reduced OT model. Overall, our proposed HOT algorithm can compute an $\varepsilon$-approximation of the KW distance between two histograms supported on $M=m \times n$ bins within $O((m^2n + n^2m)/\varepsilon)$ flops. This is the best-known computational complexity for computing an approximate KW distance to our knowledge. Moreover, we propose an efficient algorithm to recover a {transport plan} based on the obtained solution of the reduced model, which releases the power of the reduced model in applications. We implement HOT in PyTorch and extensive numerical results will be shown in Section \ref{sec: Experiment} to demonstrate the superior and robust performance of HOT for computing the KW distance, compared to state-of-the-art algorithms, including {Sinkhorn} \cite{cuturi2013sinkhorn}, convolutional {Sinkhorn}\cite{solomon2015convolutional,solomon2018optimal}, Network Simplex method \cite{goldberg1989network,gabow1991faster}, ADMM \cite{glowinski1975approximation,gabay1976dual}, interior point method (in Gurobi). 

We summarize the main contributions of this paper as follows:
\begin{enumerate}
    \item We propose an efficient HOT algorithm for solving the reduced model of the OT problem with an attractive $O(1/\varepsilon)$ iteration complexity guarantee with respect to the KKT residual.
    \item We designed a highly efficient algorithm for solving the subproblems of the HOT algorithm with linear time complexity.
    \item We propose an efficient algorithm to recover a {transport plan} based on the obtained solution of the reduced model, which removes the restriction of the reduced model in applications requiring a {transport plan}.
    \item We implement the HOT algorithm in PyTorch, which supports both CPU and GPU computation and is user-friendly for researchers in the machine learning community.
    \item Extensive numerical testings are conducted and presented to demonstrate the efficiency of the HOT algorithm.
\end{enumerate}

The rest of the paper is organized as follows. We introduce the equivalent reduced OT model in Section \ref{sec: reduced OT model}. This section also includes an efficient procedure for recovering the {transport plan} from a solution to the reduced OT problem. The HOT algorithm and its computational complexity guarantees will be presented in Section \ref{sec: HOT-algorithm}. We present extensive numerical results in Section \ref{sec: Experiment} and conclude the paper in Section \ref{sec: conclusion}.

\textbf{Notation.} We denote the \(n\)-dimensional real Euclidean space as \(\mathbb{R}^n\) and the nonnegative orthant of \(\mathbb{R}^n\) as \(\mathbb{R}^n_{+}\). For any \(x \in \mathbb{R}^n\) and \(y \in \mathbb{R}^n\), we define \(\langle x, y \rangle := \sum_{i=1}^n x_i y_i\) and \(\|x\| := \sqrt{\sum_{i=1}^n x_i^2}\), respectively. Additionally, let \(\mathbf{1}_m\) (resp. \(\mathbf{0}_m\)) denote the \(m\)-dimensional vector with all entries being 1 (resp. 0). For a given matrix \(A \in \mathbb{R}^{m \times n}\), we denote \(A^{\top} \in \mathbb{R}^{n \times m}\) as its transpose. For a collection of matrices \(\{A_1, \ldots, A_m\}\), we denote the block diagonal matrix with diagonal blocks \(A_i\) as \(\operatorname{diag}(A_1, \ldots, A_m)\). \(A_1 \otimes A_2\) stands for the Kronecker product of matrices \(A_1\) and \(A_2\). Moreover, for a closed convex set \(C\), we denote the indicator function of \(C\) and the Euclidean projector over \(C\) as \(\delta_C\) and \(\Pi_C(x) := \arg\min_{z \in C} \|x-z\|\), respectively.

\section{Kantorovich-Wasserstein distances}
\label{sec: reduced OT model}
In this section, we first introduce an equivalent reduced model of the OT problem for computing the KW distance between two-dimensional histograms. Subsequently, we present a fast and easily implementable algorithm to reconstruct the optimal transport plan of the original OT model from a solution of this reduced model.
\subsection{An equivalent reduced model of the OT problem}
In the following discussion, we assume two-dimensional histograms for simplicity. As previously mentioned, these histograms are widely used in applications as shape and image descriptors. Without loss of generality, we adopt the following assumptions and notations:
\begin{enumerate}
    \item Histograms have supports in $M = m \times n$ bins with $m$ rows and $n$ columns;
    \item The index set for bins is defined as $\mathcal{I} = \{(i, j) \mid 1 \leq i \leq m, 1 \leq j \leq n\}$. We use $(i, j)$ to denote a bin or a node corresponding to it;
    \item $\mu^1$ and $\mu^2$ are the two histograms to be compared, where each histogram $\mu^k$ is defined as $\{\mu^k_{i, j} \mid \mu^k_{i,j} \geq 0, (i, j) \in \mathcal{I},\sum\limits_{(i,j)\in \mathcal{I}} \mu^k_{i,j} = 1 \}$ for $k = 1, 2$.
\end{enumerate} 

With these notations and assumptions, the discrete OT problem for computing the KW distance between histograms $\mu^1$ and $\mu^2$ can be defined as follows:
\begin{equation}\label{model:OT}
\begin{array}{ll}
\min\limits_{\pi} & \sum\limits_{(i,j)\in \mathcal{I}}  \sum\limits_{(k,l)\in \mathcal{I}} c_{i,j;k,l} \pi_{i,j;k,l} \\
\text{s.t.} 
&\left\{\begin{aligned}
 \quad
& \sum\limits_{(k,l)\in \mathcal{I}} \pi_{i,j;k,l} = \mu^1_{i,j},\quad \forall (i,j)\in \mathcal{I}, \\
& \sum\limits_{(i,j)\in \mathcal{I}} \pi_{i,j;k,l} = \mu^2_{k,l},\quad \forall (k,l)\in \mathcal{I}, \\
& \pi_{i,j;k,l} \geq 0, \quad \forall (i,j)\in \mathcal{I}, \ \forall (k,l)\in \mathcal{I},
\end{aligned}\right.
\end{array}
\end{equation}
where $\pi$ is the transport plan between histograms $\mu^1$ and $\mu^2$. The ground distance $c_{i,j;k,l}$ is commonly defined by the $L_p^p$ distance:
\begin{equation}\label{def:c}
c_{i,j;k,l} = \|(i,j)^{\top} - (k,l)^{\top}\|_p^p = (|i-k|^p + |j-l|^p).
\end{equation}
In this paper, we focus on the case where $p=2$. By exploiting the separable structure of the ground distance, Auricchio et al. \cite{auricchio2018computing} proposed the following equivalent model in terms of the optimal objective function value:
\begin{equation}\label{model: partite-graph}
\begin{array}{ll}
\min \limits_{f^{(1)}, f^{(2)}} & \sum\limits_{(i,j)\in \mathcal{I}} \left[ \sum\limits_{k=1}^m (k-i)^2 f_{i,k,j}^{(1)} + \sum\limits_{l=1}^n (j-l)^2 f_{k,j,l}^{(2)} \right] \\
\text{s.t.} & 

\left\{\begin{aligned}
 \quad
&\sum_{i=1}^m f_{i,k,j}^{(1)} = \sum_{l=1}^n f_{k,j,l}^{(2)},  &\forall (k,j) \in \mathcal{I},\\
&\sum_{k=1}^m f_{i,k,j}^{(1)} = \mu^1_{i,j},  &\forall (i,j) \in \mathcal{I}, \\
&\sum_{j=1}^n f_{k,j,l}^{(2)} = \mu^2_{k,l}, & \forall (k,l) \in \mathcal{I}, \\
&f_{i,k,j}^{(1)}  \geq 0, f_{k,j,l}^{(2)} \geq 0,  &\forall (i,j),(k,l) \in \mathcal{I},
\end{aligned}\right.
\end{array}
\end{equation}
where $f^{(1)}_{i,k,j}$ denotes the input flow from bin $(i,j)$ to $(k,j)$ and $f^{(2)}_{k,j,l}$ denotes the output flow from bin $(k,j)$ to $(k,l)$. Compared to formulation \eqref{model:OT}, the formulation \eqref{model: partite-graph} offers substantial computational benefits. Specifically, the reduced problem \eqref{model: partite-graph} only has \( mn^2 + m^2n \) variables, whereas the original model has \( m^2n^2 \) variables. Moreover, the reduced model remains an LP problem. Consequently, popular algorithms for LP problems can be applied to solve this reduced model, such as the network-simplex method and the interior point method. Although the computation and memory costs of these mentioned algorithms are lower for the reduced model, it remains a challenge for solving large-scale problems. In this paper, we focus on addressing the challenges by designing a fast algorithm to solve the reduced model \eqref{model: partite-graph}.

To facilitate the design of the algorithm, we reformulate the model \eqref{model: partite-graph} into the following standard form of linear programming:
\begin{equation}\label{model:standLP0}
\begin{array}{ll}
{\min \limits_{x\in \mathbb{R}^{N}}} &
\langle{c}, {x}\rangle  + \delta_{\mathbb{R}_{+}^{N}}(x)\\
         \text { s.t. } &{\hat{A}} {x}=\hat{b},
\end{array}
\end{equation}
where  
\begin{enumerate}
    \item $M_3=3M-1, N=m^2n+mn^2$;
    \item $x=[ f^{(1)}; f^{(2)}]\in \mathbb{R}^{m^2n}\times\mathbb{R}^{mn^2} $ with 
    $$
    \begin{cases}
        f^{(1)}=\{f_{i,k,j}^{(1)},\ \forall(i, j) \in \mathcal{I},\ k=1,\ldots,m\}, \\
         f^{(2)}=\{f_{k,j,l}^{(2)},\ \forall(k, l) \in \mathcal{I},\ j=1,\ldots,n\}; 
    \end{cases}
     $$  
    \item $c=[c^{1};c^{2}]\in \mathbb{R}^{m^2n}\times\mathbb{R}^{mn^2}$ with 
    $$
    \begin{cases}
        c^1=\{c_{i,k,j}^{(1)}=(k-i)^2,\ \forall(i, j) \in \mathcal{I},\ k=1,\ldots,m\}, \\
        c^2=\{c_{k,j,l}^{(2)}=(j-l)^2,\ \forall(k, l) \in \mathcal{I},\ j=1,\ldots,n\}; 
    \end{cases}
    $$
     \item $\hat{b}=[\mathbf{0}_{M};\mu^1;
\mu^2]\in \mathbb{R}^{M}\times \mathbb{R}^{M} \times \mathbb{R}^{M}$;
    \item $\hat{A}=\left[\begin{array}{cc}
      {A}_1&{A}_2\\
         A_3& \mathbf{0}   \\
         \mathbf{0}&\hat{A}_4          
    \end{array}\right]  \in \mathbb{R}^{(M_3+1)\times N}$ with 
    \begin{equation}\label{def:A1A2A3}
       \begin{cases}
      {A}_1=I_{M}\otimes \mathbf{1}_{m}^{\top}\in     \mathbb{R}^{M\times m^2n},\\   
    {A}_2=-\mathbf{1}_{n}^{\top}\otimes I_{M}\in \mathbb{R}^{M\times mn^2},\\
    A_3=I_{n}\otimes (\mathbf{1}_{m}^{\top}\otimes I_{m}) \in\mathbb{R}^{M\times m^2n},\\ 
\hat{A}_4=I_{n}\otimes (\mathbf{1}_{n}^{\top}\otimes I_{m}) \in\mathbb{R}^{M\times mn^2}. \\   
     \end{cases}  
    \end{equation}  
\end{enumerate}
For notational convenience, let $\bar{I}_{m}=[I_{m-1},\mathbf{0}_{m-1}]\in \mathbb{R}^{(m-1)\times m}$. We define 
\begin{equation}\label{Mat:A&b}
 A:=\left[\begin{array}{cc}
         A_1& A_2   \\
         A_3&\mathbf{0}   \\
         \mathbf{0}&{A}_4
    \end{array}\right] \in \mathbb{R}^{M_3\times N},\quad {b}:=[0_{M};\mu^1; \bar{I}_{M}  \mu^{2}]\in \mathbb{R}^{M_3}.  
\end{equation}
with 
\begin{equation}\label{def:A4}
 {A}_4=\operatorname{diag}\left(  \mathbf{1}_{n}^{\top}\otimes I_m,\ldots,\mathbf{1}_{n}^{\top}\otimes I_m,\mathbf{1}_{n}^{\top}\otimes\bar{I}_{m}\right)\in \mathbb{R}^{(M-1)\times mn^2}.
\end{equation}
Similar to \cite[Lemma 7.1]{dantzig2003linear}, we can obtain that $A$  defined in \eqref{Mat:A&b} has full row rank, and 
$$\{x \in \mathbb{R}^N ~|~ Ax = b\} = \{x \in \mathbb{R}^N ~|~ \hat{A}x = \hat{b}\}.$$
As a result, the linear programming problem \eqref{model:standLP0} is equivalent to 
 \begin{equation}\label{model:standLP}
\begin{array}{ll}
\min\limits_{x \in \mathbb{R}^N} & \langle{c}, {x}\rangle  + \delta_{\mathbb{R}_{+}^{N}}(x)\\
         \text { s.t. } &{A} {x}={b}.
\end{array}
\end{equation}
Furthermore, the dual problem of \eqref{model:standLP} takes the form: 
\begin{equation}\label{model:dualLP}
  \underset{ y \in \mathbb{R}^{M_3}, z \in \mathbb{R}^N}{\min} \left\{-\langle{b}, {y}\rangle+\delta_{\mathbb{R}_{+}^{N}}(z)  \mid A^{\top} y+z=c\right\}.
  \end{equation}
The KKT conditions associated with \eqref{model:standLP} and \eqref{model:dualLP} can be  given by
\begin{equation}\label{def:KKT}
A^* y+z=c, \quad A x=b, \quad \mathbb{R}_{+}^{N} \ni x \perp z \in \mathbb{R}_{+}^{N},
 \end{equation}
where $x \perp z$ means $x$ is perpendicular to $z$, i.e., $\langle x, z\rangle=0$.

\subsection{Reconstruct the transport plan from the reduced model}
The absence of the transport plan \(\pi\) makes the reduced OT model less favorable in applications where the transport plan is necessary (i.e., color transfer \cite{pitie2007linear,bonneel2013example,solomon2015convolutional}). We address this issue by proposing a fast algorithm (shown in Algorithm \ref{alg:transportplan}) to reconstruct an optimal transport plan of the original model from an optimal solution of the reduced model \eqref{model: partite-graph}.
\begin{algorithm}[H]
\footnotesize 
		\caption{\footnotesize A fast algorithm for reconstructing transport plan $\pi$ from the network flows $f^{(1)}$ and $f^{(2)}$.}\label{alg:transportplan}
		\begin{algorithmic}
			\State{\textbf{Input:} An optimal flow $(f^{(1)},f^{(2)})$ of problem \eqref{model: partite-graph}.}
   	        \State{\textbf{Output:} An optimal transport plan $\pi$ of problem \eqref{model:OT}.}
                \For{$(k,j)\in\mathcal{I}$}
                    \For{$i=1,\ldots,m$}
                        \For{$l=1,\ldots,n$}
                            \State{$\pi_{i,j;k,l}=\min\{f^{(1)}_{i,k,j},f^{(2)}_{k,j,l}\}$}
                            \State{$f^{(1)}_{i,k,j}=f^{(1)}_{i,k,j}-\pi_{i,j;k,l}$}                       \State{$f^{(2)}_{k,j,l}=f^{(2)}_{k,j,l}-\pi_{i,j;k,l}$}
                        \EndFor
                    \EndFor
                \EndFor
                
		\end{algorithmic}
	\end{algorithm} 

The following proposition shows that the output of Algorithm \ref{alg:transportplan} is an optimal transport plan for the original OT model.
\begin{proposition}
Given an optimal solution \((f^{(1)},f^{(2)})\) to problem \eqref{model: partite-graph}, the output \(\pi\) of Algorithm \ref{alg:transportplan} is an optimal solution to the optimal transport problem \eqref{model:OT}.
\end{proposition}
\begin{proof}
Since \((f^{(1)},f^{(2)})\) is an optimal solution to problem \eqref{model: partite-graph}, we have
\[
f_{i,k,j}^{(1)} \geq 0, \quad f_{k,j,l}^{(2)} \geq 0, \quad \forall (i,j) \in \mathcal{I}, \quad \forall (k,l) \in \mathcal{I},
\]
which implies 
\begin{equation}\label{eq:nonnegative-Pi}
\pi_{i,j;k,l} \geq 0, \quad \forall (i,j), \quad \forall (k,l) \in \mathcal{I}.
\end{equation} 

Fix \( i,j,k \). For convenience, let \( f^{(1)-l}_{i,k,j} \) denote the value of \( f^{(1)}_{i,k,j} \) at the \( l \)-th iteration before updating. According to the update formula of \( f^{(1)}_{i,k,j} \), there must exist an index \( 1 \leq l^* \leq n \) such that \( f^{(1)-l^*}_{i,k,j} \leq f_{k,j,l^*}^{(2)} \). Otherwise, we would have  
\[
f^{(1)}_{i, k, j} > \sum_{l=1}^{n} f^{(2)}_{k, j, l},  
\]
which contradicts the feasibility of \( (f^{(1)},f^{(2)}) \):  
\[
f^{(1)}_{i,k,j} \leq \sum_{i=1}^{m} f_{i,k,j}^{(1)} = \sum_{l=1}^{n} f_{k,j,l}^{(2)}.
\]
Therefore, after \( n \) iterations, the final value of \( f^{(1)}_{i,k,j} \) must be zero, which leads to  
\begin{equation}\label{eq:Pi-1}
\sum_{l=1}^{n} \pi_{i,j;k,l} = f^{(1)}_{i,k,j}.
\end{equation}  

Similarly, by fixing \(j,k,l\), we also obtain  
\begin{equation}\label{eq:Pi-2}
\sum_{i=1}^{m} \pi_{i,j;k,l} = f^{(2)}_{k,j,l}.
\end{equation}
Hence, according to the constraints in problem \eqref{model: partite-graph}, we have 
\[
\begin{cases}
\sum\limits_{(k,l)\in \mathcal{I}} \pi_{i,j;k,l} = \sum\limits_{k=1}^m f_{i,k,j}^{(1)} = \mu^1_{i,j},\\
\sum\limits_{(i,j)\in \mathcal{I}} \pi_{i,j;k,l} = \sum\limits_{j=1}^n f_{k,j,l}^{(2)} = \mu^2_{k,l},
\end{cases}
\]
which, together with \eqref{eq:nonnegative-Pi}, shows that \(\pi\) is a feasible solution to problem \eqref{model:OT}. Furthermore, from \eqref{eq:Pi-1}, \eqref{eq:Pi-2}, and the definition of \(c\) in \eqref{def:c}, we can obtain 
\[
\begin{array}{ll}
&\sum\limits_{((i,j),(k,l))} c_{i,j;k,l} \pi_{i,j;k,l}\\ 
= & \sum\limits_{(i,j)\in \mathcal{I}} [ \sum\limits_{k=1}^m (k-i)^2 f_{i,k,j}^{(1)} + \sum\limits_{l=1}^n (j-l)^2 f_{k,j,l}^{(2)} ].
\end{array}
\]
According to \cite[Theorem 1]{auricchio2018computing}, we know that the optimal objective function values of problems \eqref{model:OT} and \eqref{model: partite-graph} are equivalent. Therefore, \(\pi\) is an optimal solution to problem \eqref{model:OT}.
\end{proof}

\begin{remark}
The worst-case computational complexity of reconstructing the transport plan via  Algorithm~\ref{alg:transportplan} is $3M^2$. In practice, we can efficiently parallelize the $(k, j)$ loop to leverage the significant benefits of GPU acceleration, enabling an efficient reconstruction of the transport plan from a solution to the reduced OT model.  Furthermore, in many applications, such as image retrieval \cite{indyk2003fast} and shape matching \cite{ling2007efficient}, only the KW distance is required. In these scenarios, the reconstruction of the transport plan is not necessary.
\end{remark}

\section{A Halpern accelerating algorithm for solving OT problem}
\label{sec: HOT-algorithm}
In this section, we first introduce an efficient Halpern accelerating method for solving problem \eqref{model:dualLP}, which includes the equivalent reduced OT problem \eqref{model:standLP} as a special case. Subsequently, we present an efficient implementation of the proposed algorithm by designing a novel procedure to solve the involved linear system in linear time complexity.
\subsection{HOT: A Halpern accelerating method for solving OT problem}
Given $\sigma>0$, the augmented Lagrange function corresponding to the dual problem \eqref{model:dualLP} is defined by, for any $(y,z,x)\in \mathbb{R}^{M_3} \times  \mathbb{R}^{N} \times  \mathbb{R}^{N}$,
$$
L_\sigma(y, z ; x):=-\langle{b}, {y}\rangle+\delta_{\mathbb{R}_{+}^{N}}(z) +\frac{\sigma}{2}\|A^{\top} y+z-c+\frac{1}{\sigma}x\|^2 -\frac{1}{\sigma}\|x\|^2.
$$
For ease of notation, denote $w:= (y, z, x)$. A fast Halpern accelerating method \cite{sun2024accelerating,halpern1967fixed} for solving OT problems is presented in Algorithm \ref{alg:HOT}. A detailed derivation of the algorithm in its current form and more discussions can be found in \cite{sun2024accelerating} and the references therein.
\begin{algorithm}[ht]
		\caption{HOT: A Halpern accelerating method for solving the OT problem \eqref{model:dualLP}.}
		\label{alg:HOT}
		\begin{algorithmic}[1]
			\State {Input: Choose an initial point $w^{0}=(y^{0}, z^{0}, x^{0})\in \mathbb{R}^{M_3} \times  \mathbb{R}^{N} \times  \mathbb{R}^{N}$. Set parameters $\sigma > 0$. For $k=0,1, \ldots,$ perform the following steps in each iteration.}          
			\State {Step 1. $\bar{y}^{k}=\underset{y \in \mathbb{Y}}{\arg \min }\left\{L_\sigma\left(y, {z}^{k} ; {x}^{k}\right)\right \}$.}
			\State{Step 2. $\bar{x}^{k}={x}^k+\sigma (A^{\top}\bar{y}^{k}+{z}^{k}-c)$.}
   		\State {Step 3. $\bar{z}^{k}=\underset{z \in \mathbb{Z}}{\arg \min }\left\{L_\sigma\left(\bar{y}^k, z ; \bar{x}^k\right)\right\}$.}     
			\State {Step 4. $w^{k+1}= \frac{1}{k+2}w^0 +  \frac{k+1}{k+2}(2\bar{w}^{k}-{w}^{k})$.}
		\end{algorithmic}
	\end{algorithm}
 
Note that Step 4 in Algorithm \ref{alg:HOT} is from the Halpern iteration with a stepsize of $\frac{1}{k+2}$. Without Step 4, the HOT algorithm reduces to the ADMM with a unit step size. According to \cite[Corollary 3.5]{sun2024accelerating}, we can obtain the global convergence of the HOT algorithm in the following proposition. The proof can be found in Appendix \ref{proof-prop2}. 

\begin{proposition}\label{prop:global-convergence}
The sequence $\{\bar{w}^k\}=$ $\{(\bar{y}^k, \bar{z}^k, \bar{x}^k)\}$ generated by the HOT algorithm in Algorithm \ref{alg:HOT} converges to the point $w^*=\left(y^*, z^*, x^*\right)$, where $\left(y^*, z^*\right)$ is a solution to problem \eqref{model:dualLP} and $x^*$ is a solution to problem \eqref{model:standLP}.
\end{proposition}
 Next, we analyze the iteration complexity of the HOT algorithm for obtaining an $\varepsilon$-approximate solution, where an appropriate measure for the quality of the solution is crucial. In this paper, we consider the residual mapping associated with the KKT system \eqref{def:KKT}:
$$
\mathcal{R}(w)=\left(\begin{array}{c}
b-Ax \\
z-\Pi_{\mathbb{R}_{+}^{N}}(z-x) \\
c-A^{\top}y- z
\end{array}\right )
$$
for any $w=(y,z,x)\in \mathbb{R}^{M_3} \times  \mathbb{R}^{N} \times  \mathbb{R}^{N}$. Note that $\mathcal{R}(w^*) = 0$ is equivalent to the facts that $x^* \in \mathbb{R}^N$ and $(y^*, z^*) \in \mathbb{R}^{M_3}\times \mathbb{R}^N$ are the solution to problems \eqref{model:standLP} and \eqref{model:dualLP}, respectively. The KKT residual $\|\mathcal{R}(\cdot)\|$ is a commonly used and practical measure for the quality of the approximation solution to \eqref{model:standLP}. It follows from \cite[Theorem 3.7]{sun2024accelerating} that the HOT algorithm enjoys an appealing $O(1/k)$ nonergodic convergence rate in terms of the KKT residual for solving \eqref{model:standLP}, which is summarized in the following proposition. The proof can be found in Appendix \ref{proof-prop-3}.
\begin{proposition}\label{prop:complexity}
   Let $\{(\bar{y}^{k},\bar{z}^{k},\bar{x}^{k})\}$ be the sequence generated by Algorithm \ref{alg:HOT}, and let $w^*=(y^*,z^*,x^*)$ be the limit point of the sequence $\{(\bar{y}^{k},\bar{z}^{k},\bar{x}^{k})\}$ and $R_0=\|x^{0}-x^{*} + \sigma (z^{0}-z^{*}) \|$. For all $k \geq 0$, we have the following bounds:
\begin{equation}\label{eq: complexity-bound-KKT}
\|\mathcal{R}(\bar w^k)\| \leq \left( \frac{\sigma +1}{{\sigma}} \right) \frac{R_0}{(k+1)}
\end{equation}	
and

\begin{equation*}
-\|z^{*}\|\frac{R_0}{ (k+1)} \leq \langle c ,\bx^k -x^{*} \rangle \leq  (\sigma\|z^{*}\|+R_0)\frac{R_0}{\sigma (k+1)} .
\end{equation*}
\end{proposition}
\begin{remark}
Note that, without acceleration, the ADMM has an \( O(1/\sqrt{k}) \) non-ergodic rate in terms of both the objective function value gap and feasibility violations \cite{davis2016convergence,cui2016convergence}. In contrast, the HOT algorithm in Algorithm \ref{alg:HOT} achieves an \( O(1/k) \) non-ergodic convergence rate, offering significant advantages for solving large-scale OT problems.  
\end{remark}

\subsection{A fast implementation of the HOT algorithm}
In this section, we present a fast implementation of the HOT algorithm. Through direct calculations, we obtain the following updates of $\bar{z}^k$ and $\bar{y}^k$ for any \(k \geq 0\):
\begin{enumerate}
    \item Update of \(\bar{z}^k\):
\[
\bar{z}^{k} = \Pi_{\mathbb{R}_{+}^{N}}\left(c - A^{\top}\bar{y}^k - \bar{x}^k / \sigma\right);
\]
\item Update of \(\bar{y}^k\):
\begin{equation}\label{update:y}
  AA^{\top} \bar{y}^k = \frac{b}{\sigma} - A\left(\frac{x^{k}}{\sigma} + z^{k} - c\right).  
\end{equation}
\end{enumerate}

Therefore, the main computational bottleneck of the HOT algorithm for solving \eqref{model:standLP} is solving the linear system \eqref{update:y}. Note that the dimension of the matrix $AA^{\top}$ is $M_3 \times M_3$ with $M_3$ defined in \eqref{model:standLP0}. In applications, $M_3$ is usually a huge number. As an illustrative example, for an image with \(256 \times 256\) pixels, \(M_3 = 196,607\). As a result, it is not computationally affordable for computing a (sparse) Cholesky decomposition for the matrix $AA^{\top}$ or solving the linear system \eqref{update:y} with standard direct solvers. Indeed, it is computationally expensive even for computing the matrix $AA^{\top}$. Instead, in the remaining part of this subsection, we will derive a linear time complexity procedure for solving the linear equation \(AA^{\top}{y} = R\) with a given vector \(R \in \mathbb{R}^{M_3}\). It is worthwhile mentioning that our procedure does not require calculating nor storing the matrix $AA^{\top}$. Note that ${A} A^{\top}$ can be written in the following form:
    \begin{equation}\label{Mat:OT_AAT}
	    {A} A^{\top}=\left[\begin{array}{ccc}
		E_{1} & E_{2} & E_{3}  \\
		 E_{2}^{\top} & E_{4} & \mathbf{0}\\ E_{3}^{\top}& \mathbf{0} &E_{5} 
	\end{array}\right],
	\end{equation}	
	where	
	\begin{enumerate}
	    \item $E_{1}=(m+n)I_{M}\in   \mathbb{R}^{M\times M}$;
	    \item $E_{2}= \operatorname{diag}\left(\mathbf{1}_{m}\mathbf{1}_{m}^{\top},\ldots, \mathbf{1}_{m}\mathbf{1}_{m} ^{\top},\mathbf{1}_{m}\mathbf{1}_{m} ^{\top} \right)  \in \mathbb{R}^{M \times M}$;       
            \item $E_{3}= -\mathbf{1}_{n}\otimes(I_m,\ldots,I_m,\bar{I}_{m}^{\top} )\in\mathbb{R}^{M\times (M-1)}$; 
	     \item $E_{4}=  m I_{M}\in   \mathbb{R}^{M\times M} $;
	     \item $E_5={A}_4{A}_4^{\top}=nI_{M-1}\in \mathbb{R}^{(M-1)\times(M-1)}$.
	\end{enumerate}
 To better explore the structure of the linear system $AA^{\top}y=R$, we rewrite it equivalently as
\begin{equation}\label{equ:normal}
	AA^{\top}y=\left[\begin{array}{ccc}
		E_{1} & E_{2} & E_{3}  \\
		 E_{2}^{\top} & E_{4} & \mathbf{0}\\ E_{3}^{\top}& \mathbf{0} &E_{5} 
	\end{array}\right]\left[\begin{array}{c}
		y_1  \\
		y_2 \\
		y_3 
	\end{array}\right]=\left[\begin{array}{c}
		R_1  \\
		R_2 \\
		R_3
	\end{array}\right],
\end{equation}
where $y:=(y_1; y_2; y_3)\in \mathbb{R}^{M}\times\mathbb{R}^{M} \times\mathbb{R}^{M-1}$ and $ R:= (R_1; R_2; R_3)\in \mathbb{R}^{M}\times\mathbb{R}^{M} \times\mathbb{R}^{M-1}$. To further explore the block structure of the linear system, we can denote $y_i:= (y_i^1; \dots; y_i^n) \in \mathbb{R}^{m} \times \cdots \times \mathbb{R}^{m}$ for $i=1,2$, and $y_3 = (y_3^1; \dots; y_3^n) \in \mathbb{R}^{m} \times \cdots \times  \mathbb{R}^{m} \times \mathbb{R}^{m-1}$. Correspondingly, we write $R_i = (R_i^1; \dots; R_i^n)$ for $i=1,2,3$. The next proposition gives an explicit formula of the solution to the linear equation in \eqref{equ:normal}. 
\begin{proposition}\label{prop:linsol_OT}
Consider \(A \in \mathbb{R}^{M_3 \times N}\) defined in \eqref{Mat:A&b}. Given \(R \in \mathbb{R}^{M_3}\), the solution \(y\) to \(AA^{\top}y = R\) in the form \eqref{equ:normal} is given by:
\begin{eqnarray}
&& y_2^{j} = \frac{1}{m}( R_2^{j} - \mathbf{1}_{m}^{\top}y^{j}_1 ), \quad j = 1, \ldots, n, \label{y2} \\
&& y_3^{j} = \frac{1}{n}( R_3^{j} + \sum_{j=1}^{n}y_{1}^{j} ), \quad j = 1, \ldots, n-1, \\
&& y_3^{n} = \frac{1}{n}( R_3^{n} + \bar{I}_{m} \sum_{j=1}^{n}y_{1}^{j} ), \label{y3} \\
&& y_1^{j} = \hat{y}_1^{j} - \hat {y}_1^{a}, \quad j = 1, \ldots, n, \label{y1}
\end{eqnarray}
where
\begin{enumerate}
    \item \(\hat {y}_1^{j} = \frac{1}{m+n} \left( \tilde{R}_{1}^{j} + \tilde{R}_{2}^{j} + \tilde{R}_{3} \right), \quad j = 1, \ldots, n,\) with \(\tilde{R}_{1}^{j} = R^{j}_1 + \frac{1}{n} \mathbf{1}_{m}^{\top} R^{j}_1\), \(\tilde{R}_{2}^{j} = -\left( \frac{1}{m} + \frac{1}{n} \right) \mathbf{1}_{m}^{\top} R^{j}_2\), and \(\tilde{R}_{3} = \frac{1}{n} \left( \sum_{j=1}^{n-1} R^{j}_{3} + \bar{I}_{m}^{\top} R^{n}_{3} \right) + \frac{1}{n^2} \mathbf{1}_{M-1}^{\top} R_3\);
    \item \(\hat {y}_1^{a} = \left( I_{m} + \frac{1}{n} \mathbf{1}_{m} \mathbf{1}_{m}^{\top} \right) \hat{W} \sum_{j=1}^{n} \hat{y}_{1}^{j}\);
    \item \(\hat{W} = \left( -\operatorname{diag}\left( \frac{1}{m} I_{m-1}, \frac{1}{m+1} \left( 1 - \frac{1}{n} \right) \right) - \frac{1}{w} d d^{\top} \right),\) with \(d = \left[ \frac{1}{m} \mathbf{1}_{m-1}; \frac{1}{m+1} \left( 1 - \frac{1}{n} \right) \right] \in \mathbb{R}^{m}\) and \(w = \frac{1}{m} - \frac{1}{(m+1)} \left( 1 - \frac{1}{n} \right)\).
\end{enumerate}
\end{proposition}

\begin{proof}
By some direct calculations, we can solve \eqref{equ:normal} equivalently as:
\begin{eqnarray}
&& y_2 = \frac{R_2 - E_2^{\top} y_1}{m}, \label{equ:normal_y_2} \\
&& y_3 = \frac{R_3 - E_3^{\top} y_1}{n}, \label{equ:normal_y_3} \\
&& \tilde{E}_1 y_1 = R_1 - \frac{1}{m} E_2 R_2 - \frac{1}{n} E_3 R_3, \label{equ:normal_2}
\end{eqnarray}
where \(\tilde{E}_1 = \left( E_1 - \frac{1}{m} E_2 E_2^{\top} - \frac{1}{n} E_3 E_3^{\top} \right)\). As a result, the key is to obtain \(y_1\) by solving \eqref{equ:normal_2}. Let \(\hat{E}_1 := E_1 - \frac{1}{m} E_2 E_2^{\top}\). By direct calculations, we have:
\begin{equation}\label{Mat:hatE1}
\hat{E}_1 = \operatorname{diag}\left( \hat{E}_1^{1}, \ldots, \hat{E}_1^{n} \right)
\end{equation}
with \(\hat{E}_1^{j} = (m+n) I_{m} - \mathbf{1}_{m} \mathbf{1}_{m}^{\top}, \quad j = 1, \ldots, n.\) By the Sherman–Morrison-Woodbury formula, we directly get:
\begin{equation}\label{equ:invhatE5}
\hat{E}_1^{-1} = \operatorname{diag}\left( (\hat{E}_1^{1})^{-1}, \ldots, (\hat{E}_1^{n})^{-1} \right)
\end{equation}
with \((\hat{E}_1^{j})^{-1} = \frac{1}{m+n} \left( I_{m} + \frac{1}{n} \mathbf{1}_{m} \mathbf{1}_{m}^{\top} \right), j = 1, \ldots, n.\)
On the other hand, let:
\[
Q = \frac{n-1}{n} I_{m} + \frac{1}{n} \bar{I}_{m}^{\top} \bar{I}_{m} \in \mathbb{R}^{m \times m}.
\]
Denote \(\hat{Q} := Q^{1/2}\) such that \(\hat{Q} \hat{Q} = Q\), and \(\bar{Q} := \mathbf{1}_{n} \otimes \hat{Q}\). We can obtain:
\[
\tilde{E}_1 = \left( \hat{E}_1 - \frac{1}{n} E_3 E_3^{\top} \right) = \hat{E}_1 - \bar{Q} \bar{Q}^{\top}.
\]
Hence, by the Sherman–Morrison-Woodbury formula, we can derive:
\begin{equation}\label{equ:invS}
\begin{array}{ll}
\tilde{E}_1^{-1} &= \left( \hat{E}_1 - \frac{1}{n} E_3 E_3^{\top} \right)^{-1} \\
&= \hat{E}_1^{-1} - \hat{E}_1^{-1} \bar{Q} W^{-1} \bar{Q}^{\top} \hat{E}_1^{-1}
\end{array}
\end{equation}
with \(W = -I_m + \bar{Q}^{\top} \hat{E}_1^{-1} \bar{Q}\). Note that:
\[
\begin{array}{ll}
W &= -I_m + \sum_{j=1}^{n} \hat{Q} (E_1^{j})^{-1} \hat{Q} \\
&= \frac{1}{m+n} \left( \operatorname{diag}\left( -m I_{m-1}, -(m+1) \right) + d_0 d_0^{\top} \right),
\end{array}
\]
where \(d_0 = [\mathbf{1}_{m-1}; \sqrt{1 - \frac{1}{n}}] \in \mathbb{R}^{m}\). Applying the Sherman–Morrison-Woodbury formula to \(W\), we can obtain:
\[
\begin{array}{ll}
W^{-1} &= (m+n) \left( \operatorname{diag}( -\frac{1}{m} I_{m-1}, -\frac{1}{m+1}) - \frac{1}{w} d_1 d_1^{\top} \right)
\end{array}
\]
with \(d_1 = [\frac{1}{m} \mathbf{1}_{m-1}; \frac{1}{m+1} \sqrt{1 - \frac{1}{n}}]\). It follows that
\begin{equation}\label{equ:QWQ}
\begin{array}{ll}
\bar{Q} W^{-1} \bar{Q}^{\top} &= \mathbf{1}_{n} \mathbf{1}_{n}^{\top} \otimes (\hat{Q} W^{-1} \hat{Q}) \\
&= \mathbf{1}_{n} \mathbf{1}_{n}^{\top} \otimes (m+n) \hat{W}.
\end{array}    
\end{equation}
To explore the block structure of \(R\), we denote:
\[
\hat{R}_{1} = R_1 - \frac{1}{m} E_2 R_2 - \frac{1}{n} E_3 R_3 = (\hat{R}_{1}^1; \ldots; \hat{R}_{1}^n) \in \mathbb{R}^{m} \times \cdots \times \mathbb{R}^{m},
\]
which implies
\[
\hat{R}_{1}^{j} = R^{j}_{1} - \frac{1}{m} \mathbf{1}_{m}^{\top} R^{j}_{2} + \frac{1}{n} \left( \sum_{j=1}^{n-1} R^{j}_{3} + \bar{I}_{m}^{\top} R^{n}_{3} \right), \quad j = 1, \ldots, n.
\]
Then, for \(j = 1, \ldots, n\),
\[
\begin{array}{ll}
\hat{y}_{1}^{j} &= (\hat{E}_{1}^{-1} \hat{R}_{1})^{j} \\
&= \frac{1}{m+n} \left( \tilde{R}_{1}^{j} + \tilde{R}_{2}^{j} + \tilde{R}_{3} \right).
\end{array}
\]
Define
\[
\hat {y}_1^{a} := \left( I_{m} + \frac{1}{n} \mathbf{1}_{m} \mathbf{1}_{m}^{\top} \right) \hat{W} \sum_{j=1}^{n} \hat{y}_{1}^{j}.
\]
From \eqref{equ:normal_2}, \eqref{equ:invS}, and \eqref{equ:QWQ}, we have:
\[
y_1^{j} = \hat{y}_1^{j} - \hat {y}_1^{a}, \quad j = 1, \ldots, n.
\]
Substituting $y_1$ into \eqref{equ:normal_y_2} and \eqref{equ:normal_y_3}, we can obtain the results for $y_2$ and $y_3$.  This completes the proof.
\end{proof}
According to the explicit formula in Proposition \ref{prop:linsol_OT}, we can immediately derive the complexity result for solving the linear equation in \eqref{equ:normal}.
\begin{corollary}\label{coro:linearsystem}
Consider \(A \in \mathbb{R}^{M_3 \times N}\) defined in \eqref{Mat:A&b}. The linear system \(AA^{\top} y = R\) in the form \eqref{equ:normal} can be solved in \(O(M_3)\) flops.
\end{corollary}
Based on this corollary, we can determine the per-iteration computational cost of the HOT algorithm in Algorithm \ref{alg:HOT} for each iteration.
\begin{corollary}\label{coro:per-cost}
The per-iteration computational complexity of the HOT algorithm in Algorithm \ref{alg:HOT} in terms of flops is \(O(N)\).
\end{corollary}
\begin{proof}
Since \(A\) has at most \(2N\) nonzero elements, the computational cost for calculating \(Ax\) and \(A^{\top}y\) is only \(O(N)\). Except for solving linear systems, which can be done in \(O(M_3)\) from Corollary \ref{coro:linearsystem}, Algorithm \ref{alg:HOT} primarily involves matrix-vector multiplications and vector additions. Hence, the computational cost of Algorithm \ref{alg:HOT} for each iteration is \(O(N)\).
\end{proof}
\begin{remark}
Leveraging the sparsity of \( A \), we can compute \( Ax \) and \( A^{\top}y \) with linear space complexity \( O(N) \). Furthermore, according to Proposition \ref{prop:linsol_OT}, the linear system \( AA^{\top}y = R \) can be solved with a memory cost of \( O(M_3) \). Hence, Algorithm \ref{alg:HOT} can be implemented with linear space complexity \( O(N) \). When \( m = n \), this corresponds to \( O(M^{1.5}) \).
\end{remark}

Combining the iteration complexity in Proposition \ref{prop:complexity} and the computational cost for each iteration in Corollary \ref{coro:per-cost}, we can derive the following overall computational complexity result of the HOT algorithm for solving problem \eqref{model: partite-graph}.

\begin{theorem}
Let \(\left\{\bar{y}^k, \bar{z}^k, \bar{x}^k\right\}\) be the sequence generated by the HOT algorithm in Algorithm \ref{alg:HOT}. For any given tolerance \(\varepsilon > 0\), the HOT algorithm needs at most
\[
\frac{1}{\varepsilon}\left(\frac{1+\sigma}{\sigma}\left(\left\|x^0-x^*\right\| + \sigma \left\|z^0-z^*\right\|\right)\right)-1
\]
iterations to return a solution to  the equivalent OT problem \eqref{model:standLP} such that the KKT residual \(\left\|\mathcal{R}\left(\bar{w}^k\right)\right\| \leq \varepsilon\), where \(\left(x^*, z^*\right)\) is the limit point of the sequence \(\left\{\bar{x}^k, \bar{z}^k\right\}\). In particular, the overall computational complexity of the HOT algorithm in Algorithm \ref{alg:HOT} to achieve this accuracy in terms of flops is
\[
O\left(\left(\frac{1+\sigma}{\sigma}\left(\left\|x^0-x^*\right\| + \sigma \left\|z^0-z^*\right\|\right)\right) \frac{m^2n+mn^2}{\varepsilon}\right).
\]
\end{theorem}
\begin{remark}
Note that the complexity of the Sinkhorn method for solving the original OT problem \eqref{model:OT} is \( O\left(\frac{m^2n^2}{\varepsilon^2}\right) \) to achieve an \(\varepsilon\)-accuracy solution in terms of objective function value \cite{dvurechensky2018computational}. Even by leveraging the separable  structure of the cost function \(c\) defined in \eqref{def:c}, the convolutional Sinkhorn method, as mentioned in \cite{solomon2015convolutional} and \cite{auricchio2018computing}, still requires \( O\left(\frac{m^2n + mn^2}{\varepsilon^2} + m^2n^2\right) \) to compute the KW distance. In contrast, the HOT algorithm exhibits lower overall computational complexity, offering a significant advantage in calculating the KW distance between large-scale histograms.
\end{remark}

\begin{remark}
Assume that \( m = n \). Blanchet et al. \cite{blanchet2024towards} demonstrated that a running time of \( \widetilde{O}({M^2}/{\varepsilon}) \) is a bottleneck in their approach by reducing an instance of the bipartite matching problem (BMP) to an OT problem. While HOT achieves a complexity of \( O(M^{1.5}/\varepsilon) \), it cannot be directly applied to solve the maximum cardinality BMP. This is because, in the reduction from BMP to the OT problem, the ground distance used in \cite{blanchet2024towards} differs from the one in our paper. Consequently, the hardness result in \cite{blanchet2024towards} does not directly contradict our complexity bound.
\end{remark}

\section{Experiments}\label{sec: Experiment}

In this section, we comprehensively compare the HOT algorithm with five other state-of-the-art methods on the popular DOTmark dataset \cite{schrieber2016dotmark}. Additionally, employing the transport plan derived from Algorithm \ref{alg:transportplan}, we test the performance of the reduced OT model on the color transfer task. The numerical results reveal that HOT provides significant advantages over state-of-the-art methods in both memory and computational efficiency, particularly for large-scale problems. More numerical results to compare the HOT algorithm to the W2NeuralDual method implemented in the OTT-JAX library \cite{cuturi2022optimal} can be found in Appendix \ref{comp_W2NeuralDual}. We also present an additional application of HOT for domain adaptation in Appendix \ref{domain-adaptation}.

\subsection{Numerical comparison on the DOTmark dataset}

The DOTmark dataset \cite{schrieber2016dotmark} is a comprehensive collection of benchmark instances for evaluating and comparing algorithms in the field of optimal transport. It consists of a variety of instances categorized into different classes, such as Classic Images, Shapes, and Gaussian Distributions. In this experiment, we selected eight images each from the Classic Images and Shapes categories, as illustrated in Fig. \ref{fig:dataset_visualize}. These selected images were resized to four different resolutions:  
$64 \times 64$, $128 \times 128$, $256 \times 256$, and $512 \times 512$. Finally, we randomly selected 10 pairs from each category and computed the KW distance for each pair.
\begin{figure*}[t!]
\centering
    \includegraphics[width=0.7\textwidth]{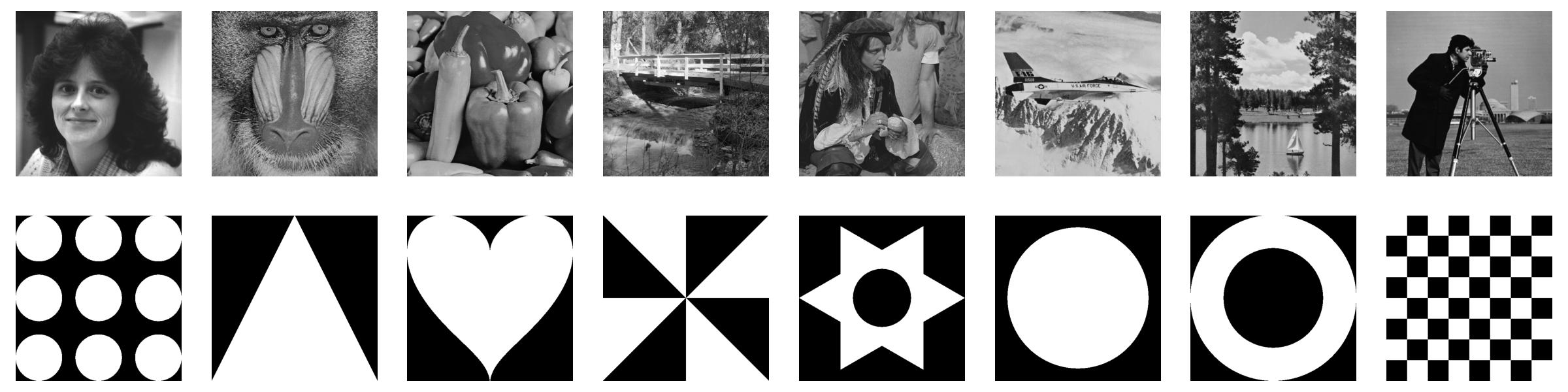}
    \caption{A visualization of the selected images from the DOTmark Dataset is presented. The upper row features images from the Classic Images category, while the bottom row contains images from the Shapes category.}
    \label{fig:dataset_visualize}
\end{figure*}

To exhibit the superiority of HOT, we compare it with the following five state-of-the-art methods:

\begin{itemize}
    \item \textbf{Sinkhorn} \cite{cuturi2013sinkhorn} is a widely used algorithm for computing an approximate KW distance by solving an entropy-regularized OT problem. As the Introduction Section discusses, its performance is sensitive to the regularization parameter, denoted as $\lambda$. To achieve varying levels of solution accuracy, we selected the $\lambda$ to be $1\%$, $0.1\%$, and $0.01\%$ of the median transport cost, following the setup from \cite{solomon2015convolutional}. Due to the potential numerical instability caused by small $\lambda$ values, we employed the Log-domain Sinkhorn algorithm implemented by POT \cite{flamary2021pot} as the baseline method.  
    
    \item \textbf{Convolutional Sinkhorn} \cite{solomon2015convolutional,solomon2018optimal} is an improved version of the Sinkhorn algorithm, optimized for computing distances over regular two-dimensional grids. Specifically, since the kernel matrix used by the Sinkhorn algorithm can be constructed using a Kronecker product of two smaller matrices, a matrix-vector product using a matrix of dimension \(M \times M\) can be replaced by two matrix-matrix products over matrices of dimension \(\sqrt{M} \times \sqrt{M}\), resulting in a significant improvement in computational efficiency. In our experiment, based on the MATLAB implementation of the convolutional Sinkhorn in \cite{auricchio2018computing}, we further developed the Log-domain convolutional Sinkhorn using PyTorch. The regularization parameter $\lambda$ is kept the same as that of the original Sinkhorn method. 
    
    \item \textbf{Gurobi (11.0.1)} is a popular optimization solver designed to address a wide range of mathematical programming problems, including linear and quadratic programming. In our experiment, we employ the interior point method (IPM) implemented in Gurobi to solve the reduced model \eqref{model:standLP}. Since it is unnecessary to obtain a basic solution, we disable the cross-over strategy.

    \item \textbf{Network Simplex} implemented in the Lemon C++ graph library \footnote{\url{https://lemon.cs.elte.hu/}}, is a highly efficient algorithm for solving uncapacitated minimum cost flow problems \cite{kovacs2015minimum}. It has demonstrated the computational advantages over the Sinkhorn type methods in \cite{auricchio2018computing}  for  
    small size
    images by solving the reduced model \eqref{model:standLP}.

    \item \textbf{ADMM} \cite{gabay1976dual,glowinski1975approximation} is a popular first-order primal-dual method for solving large-scale optimization problems. It has shown great potential in solving large-scale optimal transport problems \eqref{model:OT} in a GPU setting \cite{mai2022fast}. In this numerical experiment, we use a generalized ADMM \cite{eckstein1992douglas,xiao2018generalized} to solve the equivalent model \eqref{model:standLP} (replacing Step 4 in Algorithm \ref{alg:HOT} with \(w^{k+1}= (1-\rho)w^k+ \rho{\bar{w}}^{k}\) and setting \(\rho=1.7\)) as the baseline to evaluate the acceleration effect of the Halpern iteration.
\end{itemize}

All experiments are conducted on an Ubuntu 22.04 server equipped with an Intel(R) Xeon(R) Platinum 8480C processor and an Nvidia GeForce RTX 4090 GPU with 24 GB of RAM. Due to hardware specifications, we limit the maximum memory usage of each algorithm to 24 GB. Additionally, we set the maximum one-call running time of each algorithm to 3600 seconds. For the HOT and ADMM methods, we adopt a stopping criterion based on the relative KKT residual:
\begin{equation}
\small
    {\rm KKT_{res}} = \max\left\{\frac{\|A^{\top} y + z - c\|}{1 + \|c\|}, \frac{\|\min(x, z)\|}{1 + \|x\| + \|z\|}, \frac{\|Ax - b \|}{1 + \| b \|} \right\}.
    \label{kkt_res}
\end{equation}
Other methods use their default stopping criteria. We terminate all tested algorithms, except for Network Simplex which is an exact algorithm with a stopping tolerance of 1E-6. Finally, since different methods have varying stopping criteria, we use the following metrics to fairly evaluate the quality of the solutions: the `relative objective gap' (gap) and the `relative primal feasibility error' (feaserr). These metrics are defined as follows:
$$
\begin{array}{l}
        {\rm gap} = \frac{| \langle c, x \rangle - \langle c, x_{b} \rangle |}{| \langle c, x_{b} \rangle | + 1},\\
      {\rm  feaserr} = \max \left\{ \frac{\| \min(x, \mathbf{0}) \|}{1 + \| x \|}, \frac{\| Ax - b \|}{1 + \| b \|} \right\},
\end{array}
$$
where \(x_{b}\) is the solution obtained using Gurobi with the tolerance set to 1E-8.

We present the average results of 10 pairs within each category for all tested algorithms in Table \ref{tab:result}. Since Gurobi runs out of memory for images sized from \(256 \times 256\) to \(512 \times 512\), we only report the feaserr for the HOT, Network Simplex, and ADMM. Additionally, Table \ref{tab:result} only shows the results of Sinkhorn-type methods with \(\lambda = 0.01\% \) of the median transport cost, as this parameter returns a solution of comparable quality to other methods. For more results of the convolutional Sinkhorn with different regularization parameters, refer to Table \ref{tab:sinkhorn}. Due to space constraints, we omit the results of the Sinkhorn with varying parameters, which exhibit similar performance to the convolutional Sinkhorn in terms of the gap.

We summarized some key findings in Table \ref{tab:result} from the perspective of computational efficiency and memory cost:
\begin{enumerate}
    \item \textbf{Computational efficiency:} HOT can return a comparable solution in terms of the gap and feaserr in the shortest time.
    Although the Network Simplex and Gurobi have computational advantages for computing the KW distance for small-scale images, these methods cannot handle large-scale problems effectively. IPM implemented in Gurobi requires solving a linear equation that relies heavily on Cholesky decomposition, causing the computational cost of each iteration to increase rapidly with image size. Additionally, the inherent sequential nature of the Network Simplex algorithm makes it challenging to parallelize effectively. In contrast, HOT and convolutional Sinkhorn benefit from low per-iteration costs and are easily parallelizable. However, the convolutional Sinkhorn needs to recover the solution to the original problem \eqref{model:OT} to compute the KW distance, making it unsuitable for large-scale problems. As a result, for images sized \(128 \times 128\), HOT achieves a 17.44x speedup over Network Simplex, a 15.83x speedup over Gurobi, and a 19.54x speedup over convolutional Sinkhorn. Additionally, compared to ADMM, HOT benefits from a superior \(O(1/k)\) iteration complexity and saves 40\% of iterations for images sized \(512 \times 512\).    
    
    \item \textbf{Memory cost:} HOT demonstrates a significant advantage in memory efficiency by exploiting the sparse structure of \(A\) in \eqref{def:A1A2A3} and utilizing the explicit solution of the linear equation presented in Proposition \ref{prop:linsol_OT} to avoid sparse Cholesky decomposition. In contrast, Sinkhorn-type methods need to recover the solution to the original OT problem \eqref{model:OT} and maintain the transport cost to calculate the KW distance, which involves \(M^2\) variables. Given that these variables are stored using 64-bit floating-point representation, the Sinkhorn-based methods require at least 32 GB of memory for images sized \(256 \times 256\), which far exceeds the available 24 GB. While the IPM in Gurobi solves the reduced model \eqref{model:standLP}, each iteration requires performing a sparse Cholesky decomposition, making it unsuitable for images sized \(256 \times 256\) or larger.
\end{enumerate}

\begin{table*}
    \centering
    \caption{The numerical results of different algorithms on the DOTmark dataset. }
    \begin{tabular*}{0.85\textwidth}{ccccccccc}
        \multicolumn{3}{c}{} & \multicolumn{1}{c}{} & \multicolumn{1}{c}{}  &
        \multicolumn{1}{c}{} & \multicolumn{1}{c}{} & \multicolumn{1}{c}{} & \multicolumn{1}{c}{}    \\
        \hline
        Category    &   Resolution  &   & \textbf{HOT}  & Network Simplex  & Gurobi  & ADMM & Convolutional Sinkhorn & Sinkhorn \\
        \hline
        \multirow{14}{*}{Classic} &   \multirow{4}{*}{$64 \times 64$}     & time(s)  & \textbf{0.67} & 2.73 & 2.16 &  1.77 & 16.18 &  174.82                      \\
        &   &   gap     & 8.26E-04 & 3.46E-10 & 1.20E-04  & 2.67E-04 & 1.69E-04 & 2.12E-04 \\
        &   &   feaserr & 4.58E-07 & 4.88E-32 & 2.55E-11 & 3.09E-07 & 7.90E-07 & 9.75E-07 \\
        &   &   iter    & 1700 & - & 13 & 3420 & 64126 & 62474 \\
        \cdashline{2-9}
        &   \multirow{4}{*}{$128 \times 128$}     & time(s)  & \textbf{1.58} & 36.18 & 29.15 & 3.53 & 39.40 &  2632.17                      \\
        &   &   gap     & 6.24E-03 & 8.74E-10  & 1.07E-04  & 1.72E-03 & 6.98E-04 & 6.35E-04 \\
        &   &   feaserr & 7.27E-07 & 9.67E-32 & 7.24E-12 & 3.73E-07 & 8.34E-07 & 9.87E-07 \\
        &   &   iter    & 1170 & - & 14 & 3240 & 58446 & 57010 \\
        \cdashline{2-9}
        &   \multirow{3}{*}{$256 \times 256$}     & time(s)  & \textbf{12.98} & 2562.92 & \multirow{3}{*}{\shortstack{Memory \\ Overflow}} & 20.80 &
        \multirow{3}{*}{\shortstack{Memory \\ Overflow}} & \multirow{3}{*}{\shortstack{Memory \\ Overflow}}                        \\
        &   &   feaserr & 8.05E-07 & 1.35E-31 & & 6.04E-07 \\
        &   &   iter    & 1140 & - &  & 2250 &  &  \\
        \cdashline{2-9}
        &   \multirow{3}{*}{$512 \times 512$}     & time(s)  & \textbf{81.02} & \multirow{3}{*}{\shortstack{Over Maximum \\ Running \\ Time}} & 
        \multirow{3}{*}{\shortstack{Memory \\ Overflow}} &  116.92 & \multirow{3}{*}{\shortstack{Memory \\ Overflow}} & \multirow{3}{*}{\shortstack{Memory \\ Overflow}}    \\
        &   &   feaserr & 3.28E-07 & & & 4.32E-07 \\
        &   &   iter    & 900 & & & 1610 & & \\
        \hline
        \multirow{14}{*}{Shapes} &   \multirow{4}{*}{$64 \times 64$}     & time(s)  & \textbf{0.64} & 1.48 & 1.33 & 3.92 & 9.60 & 103.74                   \\
        &   &   gap     & 3.78E-04 & 1.81E-10 & 2.28E-05 & 5.85E-05 & 4.86E-05 & 6.07E-05 \\ 
        &   &   feaserr & 5.77E-07 & 7.24E-32 & 1.88E-10 & 2.58E-07 & 7.95E-07 & 9.68E-07 \\
        &   &   iter    & 1610 & - & 15 & 10430 & 37986 & 37077 \\
        \cdashline{2-9}
        &   \multirow{4}{*}{$128 \times 128$}     & time(s)  & \textbf{1.68} & 20.70 & 22.46 & 2.32 &  24.32 & 1616.34                       \\
        &   &   gap     & 2.51E-03 & 2.46E-09 & 2.19E-05 & 4.11E-04 & 3.28E-04 & 3.09E-04 \\
        &   &   feaserr & 1.01E-06 & 1.16E-31 & 2.01E-10 & 7.74E-07 & 8.01E-07 & 9.83E-07 \\
        &   &   iter    & 1240 & - & 18 & 2130 & 36080 & 35009 \\
        \cdashline{2-9}
        &   \multirow{3}{*}{$256 \times 256$}     & time(s)  & \textbf{14.87} & 959.77 & \multirow{3}{*}{\shortstack{Memory \\ Overflow}} &  23.30 &
        \multirow{3}{*}{\shortstack{Memory \\ Overflow}} & \multirow{3}{*}{\shortstack{Memory \\ Overflow}}                       \\
        &   &   feaserr & 6.68E-07 & 1.59E-31 & & 7.17E-07 \\
        &   &   iter    & 1310 & - & & 2530 & & \\
        \cdashline{2-9}
        &   \multirow{3}{*}{$512 \times 512$}     & time(s)  & \textbf{87.12} & \multirow{3}{*}{\shortstack{Over Maximum \\ Running \\ Time}} &
        \multirow{3}{*}{\shortstack{Memory \\ Overflow}} & 118.10 & \multirow{3}{*}{\shortstack{Memory \\ Overflow}} & \multirow{3}{*}{\shortstack{Memory \\ Overflow}}          \\
        &   &   feaserr & 3.54E-07 & & & 5.71E-07 \\
        &   &   iter    & 970 & & & 1630 & & \\
        \hline
    \end{tabular*}
    \label{tab:result}
\end{table*}

\begin{table*}
    \centering
    \caption{The numerical results of the convolutional Sinkhorn method with different $\lambda$.}
    \begin{tabular*}{0.7\textwidth}{ccccccc}
        \multicolumn{4}{c}{} & \multicolumn{1}{c}{} & \multicolumn{1}{c}{}  &
        \multicolumn{1}{c}{} \\
        \hline
        Solver & Category & Resolution & & $\lambda = 0.01\%$ & $\lambda = 0.1\%$ & $\lambda = 1\%$\\
        \hline
        \multirow{8}{*}{Convolutional Sinkhorn} & 
        \multirow{8}{*}{Classic} &
        \multirow{4}{*}{$64 \times 64$} & time(s) & 16.18 & 1.60 & 0.19 \\ 
        & & & gap & 1.69E-04 & 2.43E-02 & 3.29E-01  \\
        & & & feaserr & 7.90E-07 & 8.15E-07 & 7.10E-07 \\
        & & & iter & 64126 & 6400 & 650 \\ 
        \cdashline{3-7}
        & & \multirow{4}{*}{$128 \times 128$} & time(s) & 39.40 & 3.94 & 0.39 \\ 
        & & & gap & 6.98E-04 & 3.34E-02 & 3.51E-01  \\
        & & & feaserr & 8.34E-07 & 8.29E-07 & 7.29E-07 \\
        & & & iter & 58446 & 5850 & 594 \\ 
        \hline
    \end{tabular*}
    \label{tab:sinkhorn}
\end{table*}

To further illustrate the benefit of the explicit solution of the linear system presented in Proposition \ref{prop:linsol_OT}, we conduct a comparison between solving the linear equation using Proposition \ref{prop:linsol_OT} and sparse Cholesky decomposition. Given that the average iteration number of HOT is around 1500, we solve the linear system \eqref{equ:normal} 1500 times for different image sizes. The results are shown in Fig. \ref{fig:HOTVSSpChol}. It is clear that the computational costs of Cholesky decomposition increase rapidly as image size increases. Additionally, the forward-backward substitution method used to solve the linear system is also time-consuming because it cannot be efficiently parallelized. In contrast, solving linear systems using Proposition \ref{prop:linsol_OT} can be easily parallelized. It consumes significantly less time and remains constant as the size varies because it only requires \(O(M_3)\) flops, as shown in Corollary \ref{coro:linearsystem}. Additional numerical results on transport plan recovery time, GPU vs. CPU comparisons, and the acceleration effects of Halpern iteration for high-accuracy solutions can be found in Appendix \ref{more-DOT}.

\begin{figure}[t!]
    \includegraphics[width=0.5\textwidth]{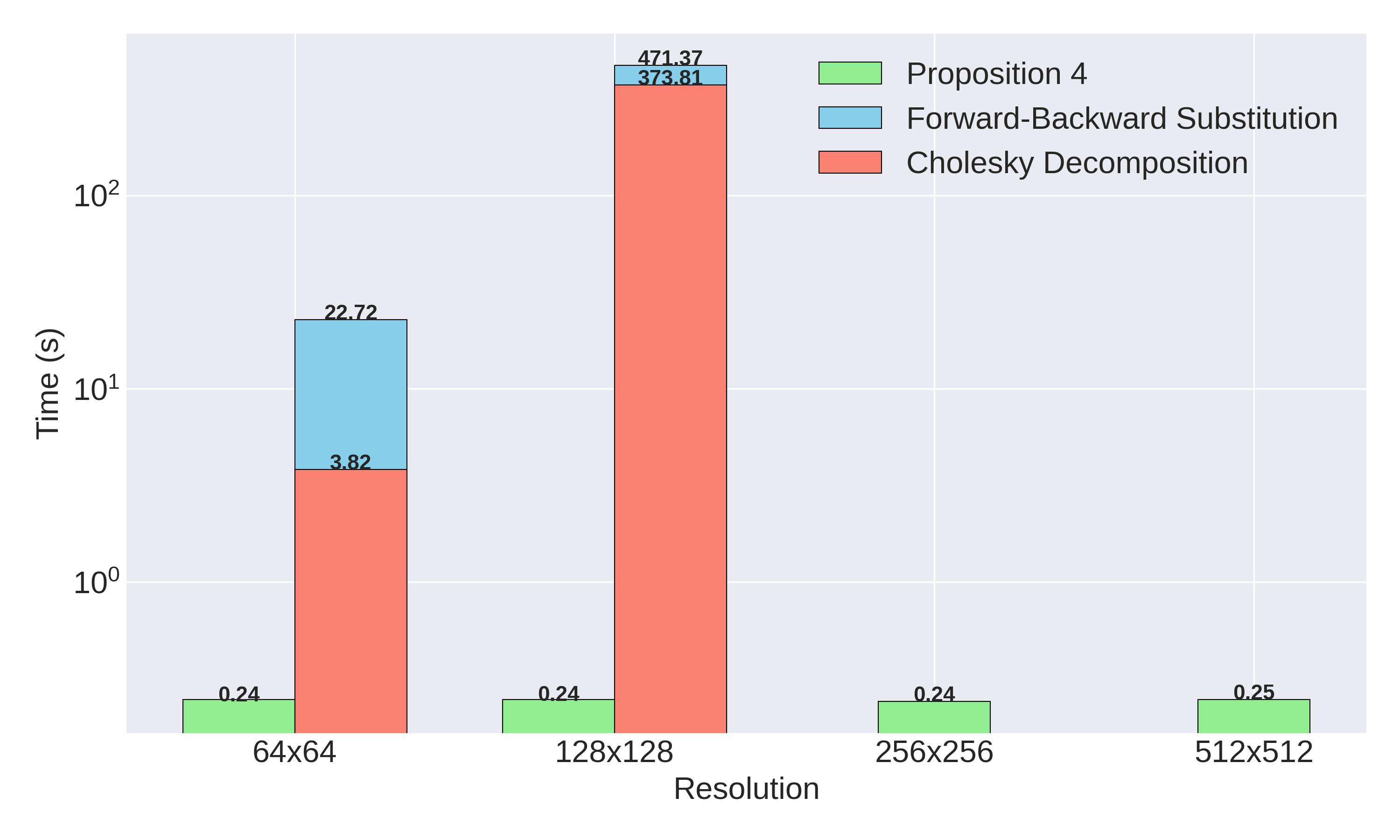}
    \caption{Comparison of solving the linear system \eqref{equ:normal} using Proposition \ref{prop:linsol_OT} and the sparse Cholesky decomposition. The time for solving the linear system using Cholesky decomposition is divided into two parts: the time for Cholesky decomposition (orange part) and the time for forward-backward substitution (blue part). For the $256 \times 256$ and $512 \times 512$ cases, Cholesky decomposition is out-of-memory in the test.} 
    \label{fig:HOTVSSpChol}
\end{figure}

\subsection{An application in color transfer}
Color transfer \cite{pitie2007linear,bonneel2013example,solomon2015convolutional} based on the OT model has found important applications in various fields, including digital image processing, computer graphics, and visual arts. It involves transferring the color characteristics from a target image to a source image to achieve a desired visual effect. Inspired by its convincing performance in color grading and color histogram manipulation \cite{bonneel2013example,solomon2015convolutional}, we conduct the color transfer over the CIE-Lab domain by applying the optimal transport model to the 1D luminance channel and the 2D chrominance channel independently. Note that the 1D optimal transport plan can be found efficiently \cite{villani2003topics}. To efficiently conduct the optimal transport over the 2D chrominance channel, we first solve the equivalent reduced OT model using the HOT Algorithm \ref{alg:HOT} and then recover an optimal transport plan using Algorithm \ref{alg:transportplan}.  
To construct the supports of the two-dimensional histograms used in the reduced OT model, we apply the K-means algorithm with $K = 128$ to all chrominance values in the CIE-Lab color space that appear in the source image or the target image. For each resulting centroid, we draw a vertical and a horizontal line. The intersections of these lines form a set of (non-uniform) bins, which serve as the histogram supports.   
The performance of color transfer for selected image pairs can be found in Fig. \ref{fig:color_transfer_head}. More results about color transfer can be found in Appendix \ref{more_application}.

\section{Conclusion}
\label{sec: conclusion}
In this paper, we proposed an efficient and scalable HOT algorithm for computing the KW distance with finite supports. In particular, the HOT algorithm solves an equivalent reduced OT model where the involved linear systems are solved by a novel procedure in linear time complexity. Consequently, we can obtain an $\varepsilon$-approximate solution to the OT problem with $M$ supports in $\mathbb{R}^2$ in $O(M^{1.5}/\varepsilon)$ flops, significantly enhancing the best-known computational complexity. Additionally, we have designed an efficient algorithm to recover an optimal transport plan from a solution to the reduced OT model, thereby overcoming the limitations of the reduced OT model in applications that require the transport plan. The extensive numerical results presented in this paper demonstrated the superior performance of the HOT algorithm. For future research directions, we aim to design an efficient implementation of the HOT algorithms for solving the OT problems with discrete supports in the more general $\mathbb{R}^d$ space. We also consider designing an efficient algorithm for solving the Wasserstein barycenter problem with discrete supports.

\section*{Acknowledgment}

The research of Yancheng Yuan was supported by the Research Center for Intelligent Operations Research and The Hong Kong Polytechnic University under the grant P0045485.  The research of Defeng Sun was supported by grants from the Research Grants Council of the Hong Kong Special Administrative Region, China (GRF Project No. 15303720) and the Research Center for Intelligent Operations Research. 

\putbib
\end{bibunit}

\ifCLASSOPTIONcaptionsoff
  \newpage
\fi

\begin{IEEEbiography}[{\includegraphics[width=0.8in,height=1in,clip,keepaspectratio]{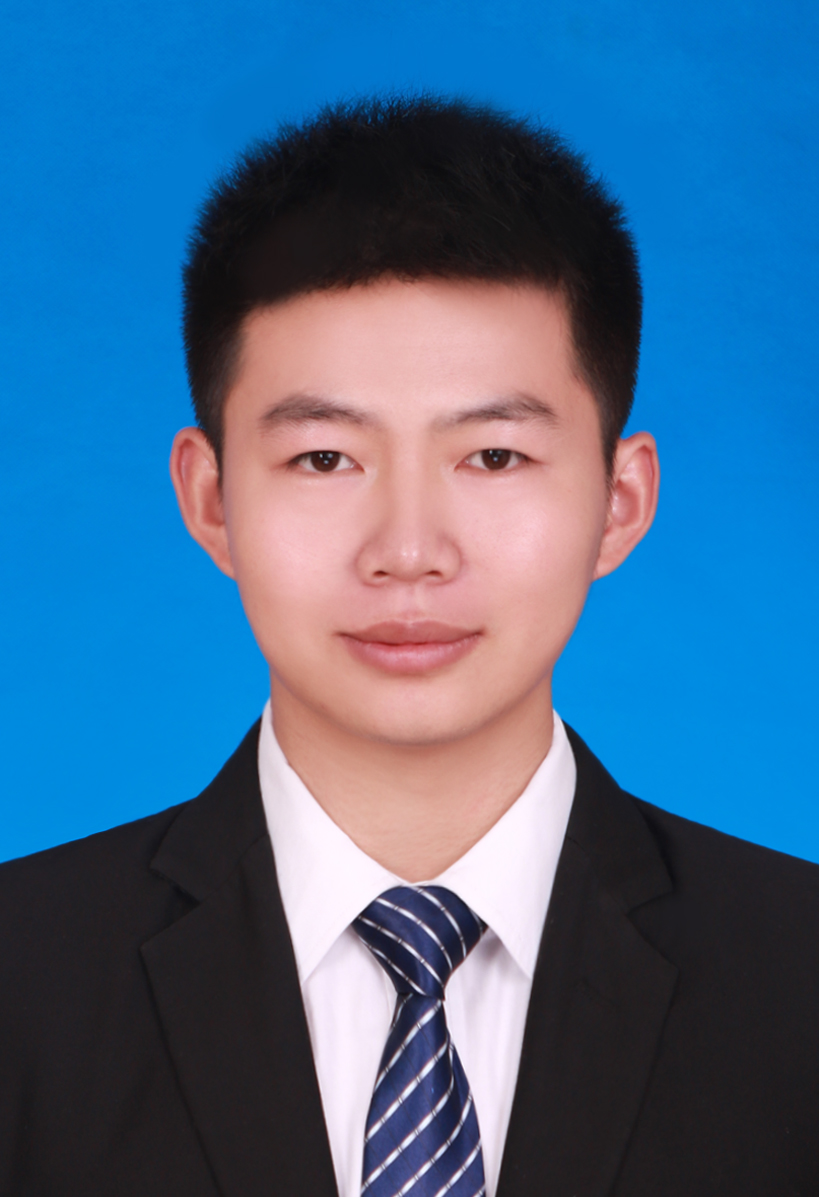}}]{Guojun Zhang} is a Ph.D. candidate in the Department of Applied Mathematics at Hong Kong Polytechnic University. His research focuses on optimization theory and complexity analysis, computational optimal transport, and software development. 
\end{IEEEbiography}

\begin{IEEEbiography}
[{\includegraphics[width=0.8in,height=1in,clip,keepaspectratio]{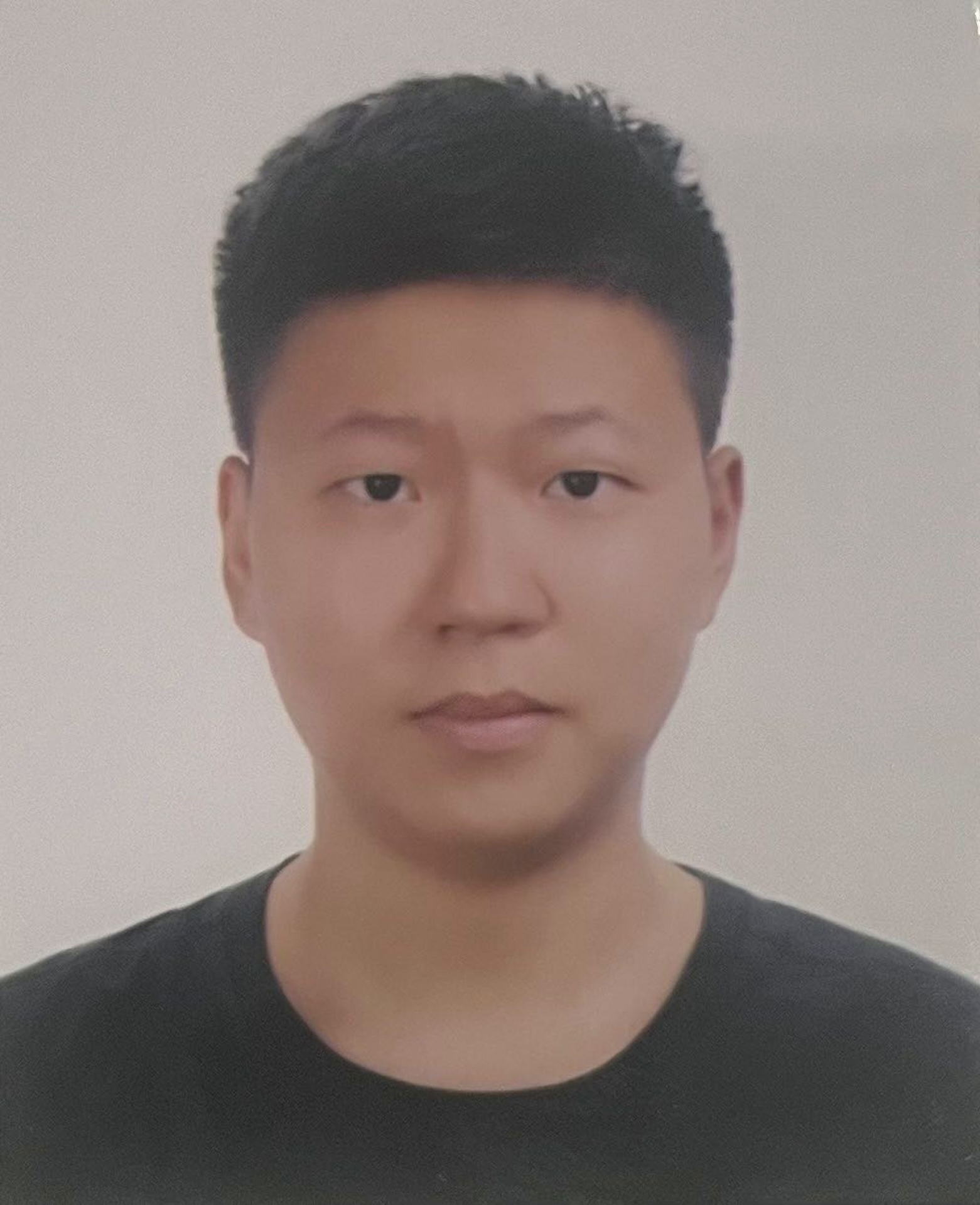}}]{Zhexuan Gu} is currently a Ph.D. student at The Hong Kong Polytechnic University. He 
received his Master's degree in the Department of Applied Mathematics at The Hong Kong Polytechnic University.
His current research interests include computational optimal transport, the foundations of artificial intelligence, and its applications in healthcare and beyond.
\end{IEEEbiography}

\begin{IEEEbiography}[{\includegraphics[width=0.8in,height=1in,clip,keepaspectratio]{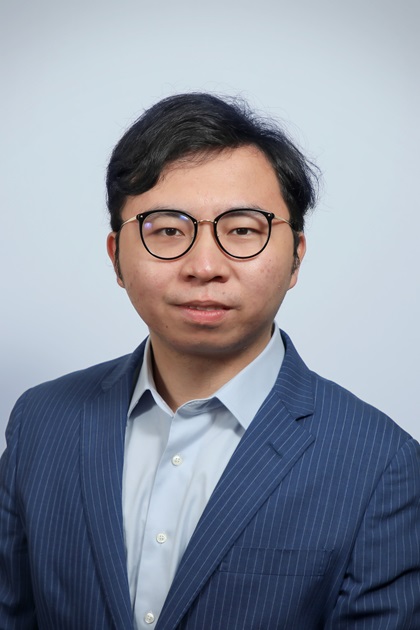}}]{Yancheng Yuan}
is an Assistant Professor at the Department of Applied Mathematics, The Hong Kong Polytechnic University. He received his PhD in Mathematics from the National University of Singapore. He was a research fellow at the NExT Research Center, National University of Singapore, mentored by Prof. Tat-Seng Chua. His research focuses on optimization theory, algorithm design and software development, the mathematical foundation of data science, and data-driven applications. His research has been published in prestigious academic journals and conferences, including Journal of Machine Learning Research, SIAM Journal on Optimization, IEEE Transactions on Neural Networks and Learning Systems, NeurIPS, ICML, WWW, SIGIR, ACL. 
\end{IEEEbiography}

\begin{IEEEbiography}[{\includegraphics[width=0.8in,height=1in,clip,keepaspectratio]{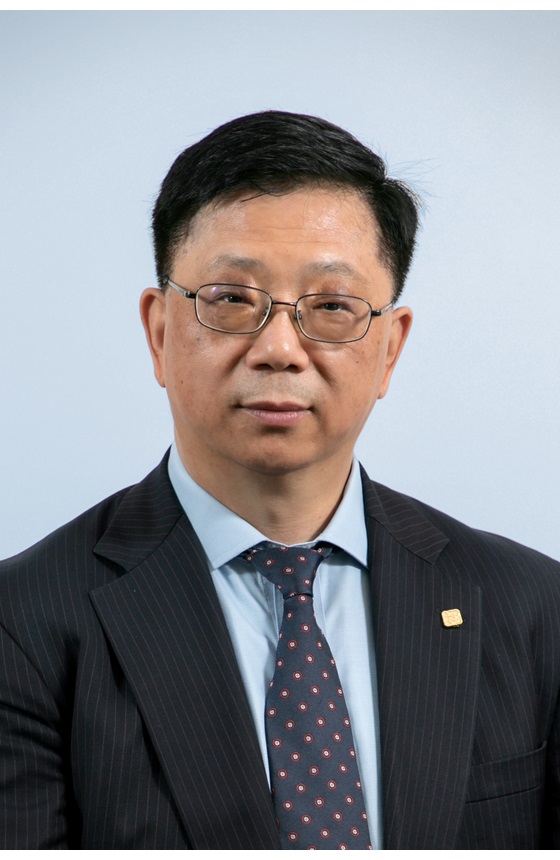}}]{Defeng Sun}
is currently Chair Professor of Applied Optimization and Operations Research at the Hong Kong Polytechnic University. He was the President of the Hong Kong Mathematical Society. He mainly publishes in continuous optimization and machine learning. Together with Professor Kim-Chuan Toh and Dr Liuqin Yang, he was awarded the triennial 2018 Beale--Orchard-Hays Prize for Excellence in Computational Mathematical Programming by the Mathematical Optimization Society. In 2020, he was elected as a Fellow of the societies CSIAM and SIAM. In 2022, he received the  RGC Senior Research Fellow Scheme award.
\end{IEEEbiography}

\newpage
\newpage

\newpage

\appendices

\twocolumn[\section*{\huge Supplementary materials for HOT: An Efficient Halpern Accelerating Algorithm for Optimal Transport Problems}
~\\

]

\begin{bibunit}

\counterwithout{table}{section}     
\setcounter{table}{0}               
\counterwithout{figure}{section}    
\setcounter{figure}{0}              
\setcounter{page}{1}

In the appendix, we first establish the equivalence between the HOT algorithm and the accelerated degenerate proximal point algorithm (dPPA) \cite{sun2024accelerating} in Appendix \ref{sec-equivlance}. Then, by analyzing the convergence and iteration complexity of the accelerated dPPA, we derive the global convergence result (Proposition \ref{prop:global-convergence}) and the complexity result (Proposition \ref{prop:complexity}) for the HOT algorithm in Appendices \ref{proof-prop2} and \ref{proof-prop-3}, respectively.  Additional numerical results, including transport plan recovery time, GPU vs. CPU comparisons, and the acceleration effects of Halpern iteration for high-accuracy solutions, are provided in Appendix \ref{more-DOT}. Further comparisons between the HOT algorithm and the W2NeuralDual method from the OTT-JAX library \cite{cuturi2022optimal} can be found in Appendix \ref{comp_W2NeuralDual}. Moreover, we also present an application of HOT for domain adaptation in Appendix \ref{domain-adaptation}. Finally, additional examples of color transfer are included in Appendix \ref{more_application}.

\section{Equivalence between the HOT and the accelerated dPPA}\label{sec-equivlance}
Let $\mathcal{N}_{\mathbb{R}_{+}^{N}}(\cdot)$ denote the normal cone of $\mathbb{R}_+^{N}$. Note that solving problems \eqref{model:standLP} and \eqref{model:dualLP} is equivalent to finding a $w^{*} \in {\mathbb{W}}=\mathbb{R}^{M_3} \times \mathbb{R}^{N} \times \mathbb{R}^{N}$ such that $0\in \cT w^{*}$, where the maximal monotone operator $\cT$ is defined by
	\begin{equation}\label{def:T}
		\cT w=\left(\begin{array}{c}
			-b +A {x}  \\
			\mathcal{N}_{\mathbb{R}_{+}^{N}}(z) +  x  \\
			c-A^{\top}{y}-{z} \\
		\end{array}     \right),  \quad  \forall w=(y,z,x)\in \mathbb{W}.
	\end{equation}
 Consider the following self-adjoint linear operator $\cM: \mathbb{W}  \rightarrow \mathbb{W}$,
	\begin{equation}\label{def:M}
		\cM=\left[\begin{array}{ccc}
			0  &\quad  0  \quad &\quad  0 \\
			0 &\quad \sigma I_{N}  & I_{N} \\
			0 &\quad  I_{N} & \quad \frac{1}{\sigma} I_{N}
		\end{array}\right],
\end{equation}
where $I_{N}$ denotes the identity matrix in $\mathbb{R}^{N\times N}$. We can establish the equivalence between the HOT algorithm and the accelerated dPPA \cite{sun2024accelerating} in the following proposition.

\begin{proposition}\label{prop:equ-acc-PADMM-acc-dPPM}
 Consider {the operators} $\cT$ defined in \eqref{def:T} and $\cM$ defined in \eqref{def:M}. Then the sequence $\left\{w^k\right\}$ generated by the HOT algorithm in Algorithm \ref{alg:HOT} coincides with the sequence $\left\{w^k\right\}$ generated by the following accelerated dPPA for any $k\geq 0$,
\begin{equation}\label{alg:acc-dPPA}
  \left\{\begin{array}{l}
    \displaystyle \bw^{k} =\hcT w^k=  (\cM+\cT)^{-1}\cM w^{k},  \\
    \displaystyle \hw^{k} =\hcF w^k=  2\bw^{k}-{w}^{k} ,  \\
    \displaystyle w^{k+1}=\frac{1}{k+2} w^{0}+\frac{k+1}{k+2}\hw^{k},
\end{array}\right.  
\end{equation}
with the same initial point $ w^0 \in \mathbb{W}$. Additionally, $\cM$ is an admissible preconditioner\footnote{In \cite{bredies2022degenerate}, an admissible preconditioner for the operator $\cT: \mathbb{W} \rightarrow 2^{\mathbb{W}}$ is a linear, bounded, self-adjoint, and positive semidefinite operator $\cM: \mathbb{W}\rightarrow\mathbb{W}$ such that $(\cM+\cT)^{-1}\cM$ is single-valued and has full domain.} such that $(\cM + \cT )^{-1}$ is Lipschitz continuous.
\end{proposition}
\begin{proof}
We establish that \( (\cM+\cT)^{-1} \) is single-valued by contradiction. Suppose, for the sake of contradiction, that \( (\cM+\cT)^{-1} \) is not single-valued. Then, there exist distinct points \( \bw_1 = (\by_1, \bz_1, \bx_1) \in \mathbb{W} \) and \( \bw_2 = (\by_2, \bz_2, \bx_2) \in \mathbb{W} \) such that  
\[
\bw_1, \bw_2 \in (\cM+\cT)^{-1} v
\]
for some \( v = (v_y, v_z, v_x) \in \mathbb{W} \). This implies that for \( i = 1,2 \), the following conditions hold:  
\begin{eqnarray}
    && v_y = -b + A \bx_i, \label{lem:inveritble-1}\\
    && v_z \in \mathcal{N}_{\mathbb{R}_{+}^{N}}(\bz_i) + \sigma \bz_i + 2\bx_i, \label{lem:inveritble-2}\\
    && v_x = c - A^{\top} \by_i + \sigma^{-1} \bx_i. \label{lem:inveritble-3}
\end{eqnarray}
By direct calculations, we obtain:  
\begin{eqnarray}
    && A(\bx_1 - \bx_2) = 0, \label{lem:inveritble-4}\\
    && \langle -\sigma(\bz_1 - \bz_2) - 2(\bx_1 - \bx_2), \bz_1 - \bz_2 \rangle \geq 0, \label{lem:inveritble-5}\\
    && \sigma A^{\top}(\by_1 - \by_2) = \bx_1 - \bx_2, \label{lem:inveritble-6}
\end{eqnarray}
where inequality \eqref{lem:inveritble-5} follows from the monotonicity of \( \mathcal{N}_{\mathbb{R}_{+}^{N}}(\cdot) \).  Substituting \eqref{lem:inveritble-6} into \eqref{lem:inveritble-4}, we obtain  
\[
\sigma A A^{\top} (\by_1 - \by_2) = 0.
\]
Since \( A A^{\top} \) is invertible, it follows that  
\begin{equation}\label{lem:inveritble-7}
    \by_1 - \by_2 = 0.
\end{equation}
Substituting \eqref{lem:inveritble-7} into \eqref{lem:inveritble-6}, we deduce that  
\begin{equation}\label{lem:inveritble-8}
    \bx_1 - \bx_2 = 0.
\end{equation}
Similarly, we conclude from  \eqref{lem:inveritble-5} and \eqref{lem:inveritble-8} that  
\begin{equation}\label{lem:inveritble-9}
    \bz_1 - \bz_2 = 0.
\end{equation}
Thus, equations \eqref{lem:inveritble-7}, \eqref{lem:inveritble-8}, and \eqref{lem:inveritble-9} imply that \( \bw_1 = \bw_2 \), contradicting our assumption. Therefore, we conclude that \( (\cM+\cT)^{-1} \) is single-valued.  

To establish the Lipschitz continuity of \( (\cM+\cT)^{-1} \), consider \( \bw_i = (\by_i, \bz_i, \bx_i) \in \mathbb{W} \) such that  
\[
\bw_i = (\cM+\cT)^{-1} v_i, \quad \text{where} \quad v_i = (v_{iy}, v_{iz}, v_{ix}), \quad i=1,2.
\]
Following the derivations in \eqref{lem:inveritble-1}–\eqref{lem:inveritble-3}, we obtain for \( i = 1,2 \):
\begin{eqnarray}
    && v_{iy} = -b + A \bx_i, \label{lem:inveritble-iy}\\
    && v_{iz} \in \mathcal{N}_{\mathbb{R}_{+}^{N}}(\bz_i) + \sigma \bz_i + 2\bx_i,\\
    && v_{ix} = c - A^{\top} \by_i + \sigma^{-1} \bx_i. \label{lem:inveritble-ix}
\end{eqnarray}
Substituting \eqref{lem:inveritble-ix} into \eqref{lem:inveritble-iy} and applying \cite[Corollary 23.5.1]{rockafellar1970convex}, we obtain:
\begin{eqnarray*}
    && \by_i = (AA^{\top})^{-1} \left( \frac{v_{iy} + b}{\sigma} - A(v_{ix} - c) \right),\\
    && \bz_i = (\mathcal{N}_{\mathbb{R}^{N}_+}(\cdot) + \sigma I_{N})^{-1} (v_{iz} - 2\bx_i).
\end{eqnarray*}
Thus, there exists a positive constant \( L_1 \) such that  
\begin{equation}\label{lem:inveritble-11}
    \begin{aligned}
        \|\by_1 - \by_2\| &= \Big\|(AA^{\top})^{-1} \big( \frac{v_{1y} - v_{2y}}{\sigma} - A (v_{1x} - v_{2x}) \big) \Big\| \\
        &\leq L_1 (\|v_{1y} - v_{2y}\| + \|v_{1x} - v_{2x}\|).
    \end{aligned}
\end{equation}
It follows from \eqref{lem:inveritble-ix} and \eqref{lem:inveritble-11} that there exists a positive constant \( L_2 \) such that  
\begin{equation}\label{lem:inveritble-12}
    \begin{aligned}
        \|\bx_1 - \bx_2\| &= \|\sigma (v_{1x} - v_{2x} + A^{\top}(\by_1 - \by_2))\| \\
        &\leq L_2 (\|v_{1x} - v_{2x}\| + \|v_{1y} - v_{2y}\|).
    \end{aligned}
\end{equation}
Similarly, since the resolvent \( (\mathcal{N}_{\mathbb{R}^{N}_+}(\cdot) + \sigma I_{N})^{-1} \) is non-expansive \cite{rockafellar1998variational}, there also exists a positive constant \( L_3 \) such that  
\begin{equation}\label{lem:inveritble-13}
    \|\bz_1 - \bz_2\| \leq L_3 (\|v_{1y} - v_{2y}\| + \|v_{1z} - v_{2z}\| + \|v_{1x} - v_{2x}\|).
\end{equation}
Hence, combining \eqref{lem:inveritble-11}, \eqref{lem:inveritble-12}, and \eqref{lem:inveritble-13}, we conclude that there exists a positive constant \( L \) such that  
\[
\|\bw_1 - \bw_2\| \leq L \|v_1 - v_2\|,
\]
which establishes the Lipschitz continuity of \( (\cM+\cT)^{-1} \).

Finally, following the proof establishing the equivalence between the semi-proximal alternating direction method  
of multipliers (spADMM) and the (partial) PPA, as outlined in Appendix B of \cite{yang2025accelerated},  
we obtain that the sequence \( \{w^k\} \) generated by Algorithm \ref{alg:HOT} coincides with the sequence \( \{w^k\} \)  
produced by the following scheme:
\begin{equation*}
    \cM w^k  \in (\cM+\cT) \bw^k, \quad 
    w^{k+1}=\frac{1}{k+2} w^{0}+\frac{k+1}{k+2}(2\bw^{k}-{w}^{k}).
\end{equation*}
Since \( (\cM+\cT)^{-1} \) is single-valued from the previous proof, it follows that
\begin{equation*}
    \bw^k = (\cM+\cT)^{-1} \cM w^k, \quad 
    w^{k+1} = \frac{1}{k+2} w^{0} + \frac{k+1}{k+2}(2\bw^{k} - {w}^{k}).
\end{equation*}
For an arbitrary choice of \( w^k \in \mathbb{W} \), each step in Algorithm \ref{alg:HOT} is well-defined.  
Thus, based on the established equivalence, we conclude that \( (\cM+\cT)^{-1} \cM \) has full domain.  
Combining this result with the Lipschitz continuity proven earlier, we deduce that \( \cM \) is an admissible  
preconditioner such that \( (\cM + \cT )^{-1} \) is Lipschitz continuous, completing the proof.
\end{proof}

Let \( \mathcal{C}^{\top} := (0, \sqrt{\sigma} I_{N}, \frac{1}{\sqrt{\sigma}} I_N) \). It is straightforward to verify that \( \mathcal{M} = \mathcal{C} \mathcal{C}^{\top} \). Define  
\begin{equation}\label{def:tcT&tcF}
    \tcT := \mathcal{C}^{\top} (\mathcal{M} + \cT)^{-1} \mathcal{C}, \quad 
    \tcF := 2\tcT - I_{\mathbb{W}},
\end{equation}
where \( I_{\mathbb{W}} \) denotes the identity operator on \( \mathbb{W} \). The following proposition summarizes some key properties of \( \tcT \) and \( \tcF \).

\begin{proposition}\label{prop:tcT}	
Consider the operators \( \cT \) defined in \eqref{def:T} and \( \cM \) defined in \eqref{def:M}. Then, the operator \( \widetilde{\cT} \) in \eqref{def:tcT&tcF} is everywhere well-defined and firmly nonexpansive. Moreover, the operator  $\tcF = 2\tcT - I_{\mathbb{W}}$ is nonexpansive. Furthermore, we have the equivalence  
\[
\mathcal{C}^{\top} \cT^{-1}(0) = \mathcal{C}^{\top} \operatorname{Fix} \hcT = \operatorname{Fix} \tcT = \operatorname{Fix} \tcF,
\]
where \( \operatorname{Fix} \hcT \) denotes the set of fixed points of the operator \( \hcT \).
\end{proposition}
\begin{proof}
The firm nonexpansiveness of \( \tcT \) follows from Theorem 2.13 in \cite{bredies2022degenerate}. Furthermore, by \cite[Proposition 4.4]{bauschke2017convex}, \( \tcF \) is nonexpansive. Finally, from the proof of Theorem 2.14 in \cite{bredies2022degenerate}, we obtain  
\[
\mathcal{C}^{\top} \operatorname{Fix} \hcT = \operatorname{Fix} \tcT.
\]  
Since \( \cT^{-1}(0) = \operatorname{Fix} \hcT \), the result follows.
\end{proof}

To analyze the global convergence and iteration complexity of the HOT algorithm (Algorithm \ref{alg:HOT}),  
we introduce two shadow sequences \( \{u^k\} \) and \( \{\bu^k\} \), defined as  
\begin{equation}\label{alg:u}
    u^{k} := \cC^{\top}w^{k}, \quad \bu^k := \cC^{\top}\bw^{k}, \quad \forall k \geq 0,
\end{equation}
where \( \{w^k\} \) and \( \{\bw^k\} \) are the sequences generated by Algorithm \ref{alg:HOT}.  
Applying Proposition \ref{prop:equ-acc-PADMM-acc-dPPM}, we obtain the following identity:
\begin{equation}\label{alg:u-acc}
    u^{k+1} = \frac{1}{k+2} u^{0} + \frac{k+1}{k+2} \tcF u^k, \quad \forall k \geq 0.
\end{equation}
We are now ready to prove Proposition \ref{prop:global-convergence}.

\section{Proof of Proposition \ref{prop:global-convergence}}\label{proof-prop2}
\begin{proof}
Note that the scheme in \eqref{alg:u-acc} is the Halpern iteration applied to the nonexpansive operator $\tcF$. It follows from the global convergence of the Halpern iteration in \cite[Theorem 2]{wittmann1992approximation} that
		\begin{equation}\label{Th:convergence-u}
			u^{k}\rightarrow u^*,
		\end{equation}
where $u^*$ is a point in $\operatorname{Fix} \tcF$. By utilizing Proposition \ref{prop:tcT}, we have
		$$
		\mathcal{C}^{\top}\mathcal{T}^{-1}(0)=\mathcal{C}^{\top} \operatorname{Fix} \widehat{\mathcal{T}}=\operatorname{Fix} \tcT = \operatorname{Fix} \tcF,
		$$
		which implies that there exists a $w^{*}$ in $\mathcal{T}^{-1}(0)$ such that $\mathcal{C}^{\top}w^{*}=u^*$. Consequently, by the relationship between $\{u^k\}$ and $\{w^k\}$ in \eqref{alg:u}, and \eqref{Th:convergence-u}, we can obtain
		\begin{equation}\label{Th:convergence-barw}
        \begin{array}{ll}
           	\bar{w}^k&=(\mathcal{M}+\mathcal{T})^{-1}\mathcal{C}\mathcal{C}^{*}w^{k}=(\mathcal{M}+\mathcal{T})^{-1}\mathcal{C} u^{k} \\
            &\rightarrow (\mathcal{M}+\mathcal{T})^{-1} \mathcal{C}\mathcal{C}^{\top}w^{*}=(\mathcal{M}+\mathcal{T})^{-1}\cM w^{*}=w^{*},
        \end{array}
		\end{equation}
		where the continuity of $(\mathcal{M}+\mathcal{T})^{-1}\mathcal{C}$ is derived from the composition of a continuous function $(\mathcal{M}+\mathcal{T})^{-1}$ showed in Proposition \ref{prop:equ-acc-PADMM-acc-dPPM} and a linear operator $\mathcal{C}$. Hence, $\{\bar{w}^k\}$ converges to $w^*$, which completes the proof. 
\end{proof}

\section{Proof of Proposition \ref{prop:complexity}}\label{proof-prop-3}
\begin{proof}
The shadow sequence \( \{u^k\} \) satisfying \eqref{alg:u-acc} corresponds exactly to the Halpern iteration. By Proposition \ref{prop:tcT}, we know that \( \tcF \) is nonexpansive. Applying \cite[Theorem 2.1]{lieder2021convergence}, we obtain  
\begin{equation}\label{prop:complexity-acc-dPPM-1}
    \|u^{k} - \tcF u^k\| \leq \frac{2\|u^0 - u^*\|}{k+1},  
    \quad \forall k \geq 0, \quad u^* \in \operatorname{Fix} \tcF.
\end{equation}
Furthermore, by Proposition \ref{prop:tcT}, we have \( \operatorname{Fix} \tcF = \cC^{*} \cT^{-1}(0) \).  
Thus, for any \( u^* \in \operatorname{Fix} \tcF \), there exists a point \( w^{*} \in \cT^{-1}(0) \) such that \( \cC^{*} w^{*} = u^{*} \).  
Substituting this into \eqref{prop:complexity-acc-dPPM-1}, we obtain, for all \( k \geq 0 \) and \( w^* \in \cT^{-1}(0) \),
\begin{equation*}
    \|\cC^{*} w^{k} - \cC^{*} \hcF w^k\| \leq \frac{2\|\cC^{*} w^{0} - \cC^{*} w^{*}\|}{k+1},  
\end{equation*}
which implies
\begin{equation}\label{complexity:acc-dPPA}
    \|w^{k} - \hat{w}^{k}\|_{\cM} = \|w^{k} - \hcF w^k\|_{\cM} 
    \leq \frac{2\|w^{0} - w^{*}\|_{\cM}}{k+1}
\end{equation}
with the seminorm  defined by  $\|w\|_{\mathcal{M}} := \sqrt{\langle w, w\rangle_{\mathcal{M}}} = \sqrt{\langle w, \cM w\rangle}$.

We now estimate the convergence rate of \( \mathcal{R}(\bw^k) \) for any \( k \geq 0 \). From \eqref{complexity:acc-dPPA}, we obtain  
\[
\|\hat{w}^{k} - w^{k}\|_{\cM}^{2} \leq \frac{4\|w^0-w^*\|_{\cM}^2}{(k+1)^2}, \quad \forall k \geq 0.
\]
By the definition of \( \cM \) in \eqref{def:M}, this can be rewritten as  
\begin{equation}\label{Th:complexity-acc-pADMM-1}
    \frac{1}{\sigma} \|\sigma (\hat{z}^{k} - z^{k}) + (\hx^{k} - x^{k})\|^{2} \leq \frac{4R_0^2}{ \sigma(k+1)^2}, \quad \forall k \geq 0.
\end{equation}
From Step 2 of the accelerated dPPA scheme in \eqref{alg:acc-dPPA}, we obtain, for any \( k \geq 0 \),
\begin{equation*}
    \begin{cases}
        \hy^{k} - y^{k} = 2(\by^{k} - y^{k}), \\
        \hz^{k} - z^{k} = 2(\bz^{k} - z^{k}), \\
        \hx^{k} - x^{k} = 2(\bx^{k} - x^{k}).
    \end{cases}
\end{equation*}
Thus, we can rewrite \eqref{Th:complexity-acc-pADMM-1} as  
\begin{equation}\label{Th:complexity-acc-pADMM-2}
    \frac{1}{\sigma} \|\sigma (\bz^{k} - z^{k}) + (\bx^{k} - x^{k})\|^{2} \leq \frac{R_0^2}{\sigma(k+1)^2}, \quad \forall k \geq 0.
\end{equation}
From Step 2 of Algorithm \ref{alg:HOT}, we deduce that for any \( k \geq 0 \),
\begin{equation*}
    \begin{array}{ll}
        &\|\sigma (\bz^{k} - z^{k}) + (\bx^{k} - x^{k})\| \\
        &= \|\sigma (\bz^{k} - z^{k}) + \sigma (A^{\top} \by^{k} + z^{k} - c)\| \\
        &= \sigma \|A^{\top} \by^{k} + \bz^{k} - c\|,
    \end{array}
\end{equation*}
which, combined with \eqref{Th:complexity-acc-pADMM-2}, yields  
\begin{equation}\label{Th:complexity-acc-pADMM-3}
    \|A^{\top} \by^{k} + \bz^{k} - c\| \leq \frac{R_0}{{\sigma}(k+1)}, \quad \forall k \geq 0.
\end{equation}
Moreover, from the optimality conditions of the subproblems in Algorithm \ref{alg:HOT}, we obtain, for any \( k \geq 0 \),
\begin{equation}\label{Th:complexity-acc-pADMM-4}
    \begin{cases}
        AA^{\top} \bar{y}^k = \frac{b}{\sigma} - A \left(\frac{x^{k}}{\sigma} + z^{k} - c\right), \\ 
        \bar{z}^{k} = \Pi_{\mathbb{R}_{+}^{N}}\left( \bz^{k} - \bx^k - \sigma (A^{\top} \bar{y}^k + \bz^k - c)\right).
    \end{cases}
\end{equation}
Thus, from Step 2 of Algorithm \ref{alg:HOT} and \eqref{Th:complexity-acc-pADMM-4}, we have  
\begin{equation}\label{Th:complexity-acc-pADMM-5}
    \begin{array}{ll}
        &\|b - A \bx^{k}\| \\
        &= \| b - A (x^{k} + \sigma (A^{\top} \by^k + z^k - c)) \| \\
        &= \| b - A (x^{k} + \sigma (z^k - c)) - (b - A (x^{k} + \sigma (z^k - c))) \| \\
        &= 0.
    \end{array}
\end{equation}
Similarly, from \eqref{Th:complexity-acc-pADMM-3} and \eqref{Th:complexity-acc-pADMM-4}, we obtain, for any \( k \geq 0 \),
\begin{equation}\label{Th:complexity-acc-pADMM-6}
    \begin{array}{ll}
        \|\bz^{k} - \Pi_{\mathbb{R}_{+}^{N}}(\bz^k - \bx^k)\| 
        &\leq \sigma \|A^{\top} \by^{k} + \bz^{k} - c\| \\
        &\leq \frac{R_0}{(k+1)}.
    \end{array}
\end{equation}
Therefore, by \eqref{Th:complexity-acc-pADMM-3}, \eqref{Th:complexity-acc-pADMM-5}, and \eqref{Th:complexity-acc-pADMM-6}, we conclude that for any \( k \geq 0 \),
\begin{equation*}
        \| \mathcal{R}(\bar{w}^k)\| \leq   \frac{\sigma + 1}{\sigma} \frac{R_0}{(k+1)}.
\end{equation*}

We now estimate the objective errors. For any \( k \geq 0 \), define  
\begin{equation}\label{def:tx}
    \tilde{x}^k = \bx^k + \sigma (A^{\top} \by^k + \bz^k - c).
\end{equation}
From the second equation in \eqref{Th:complexity-acc-pADMM-4}, it follows that  
\begin{equation}\label{def:tx-prop}
    \tilde{x}^k \in \mathbb{R}^{N}_+, \quad \text{and} \quad \langle \tilde{x}^k, \bz^k \rangle = 0. 
\end{equation}
Thus, by the convexity of \( \delta_{\mathbb{R}^{N}_+}(\cdot) \) and the KKT system \eqref{def:KKT}, we obtain  
\[
\delta_{\mathbb{R}^{N}_+}(\tilde{x}^k) \geq \delta_{\mathbb{R}^{N}_+}(x^{*}) + \langle -z^{*}, \tilde{x}^k - x^* \rangle.
\]
Adding \( \langle c, \bx^k - x^{*} \rangle \) to both sides and applying \eqref{Th:complexity-acc-pADMM-3}, \eqref{Th:complexity-acc-pADMM-5}, and \eqref{def:tx}, we derive  
\[
\begin{aligned}
    \langle c, \bx^k - x^{*} \rangle 
    &\geq \langle c, \bx^k - x^{*} \rangle + \langle -z^{*}, \tilde{x}^k - x^* \rangle \\
    &= \langle c - z^{*}, \bx^k - x^{*} \rangle + \langle -z^{*}, \sigma (A^{\top} \by^k + \bz^k - c) \rangle \\
    &= \langle A^{\top} y^*, \bx^k - x^{*} \rangle + \langle -z^{*}, \sigma (A^{\top} \by^k + \bz^k - c) \rangle \\
    &= \langle -z^{*}, \sigma (A^{\top} \by^k + \bz^k - c) \rangle \\
    &\geq -\|z^{*}\| \frac{R_0}{k+1}.
\end{aligned}
\]
Similarly, from the second equation in \eqref{Th:complexity-acc-pADMM-4}, we obtain  
\[
-\bz^k \in \mathcal{N}_{\mathbb{R}^{N}_+}(\tilde{x}^k),
\]
which implies  
\[
\delta_{\mathbb{R}^{N}_+}(x^{*}) \geq \delta_{\mathbb{R}^{N}_+}(\tilde{x}^{k}) + \langle -\bz^{k}, x^{*} - \tilde{x}^k \rangle.
\]
Adding \( \langle c, x^* - \bx^{k} \rangle \) to both sides and using \eqref{def:tx-prop}, we obtain  
\begin{equation}\label{Th:complexity-acc-pADMM-7}
    \begin{aligned}
        \langle c, x^* - \bx^{k} \rangle 
        &\geq \langle -\bz^k, x^{*} - \tilde{x}^k \rangle + \langle c, x^* - \bx^{k} \rangle \\
        &= \langle -\bz^k, x^{*} \rangle + \langle c, x^* - \bx^{k} \rangle \\
        &= \langle c - \bz^k, x^{*} \rangle - \langle c, \bx^{k} \rangle.
    \end{aligned}
\end{equation}
For convenience, let  
\[
\Delta^k = c - (A^{\top} \by^k + \bz^k).
\]
Using \eqref{Th:complexity-acc-pADMM-7} and \eqref{Th:complexity-acc-pADMM-5}, we obtain  
\[
\begin{aligned}
    \langle c, \bx^k - x^{*} \rangle 
    &\leq - \langle \Delta^k + A^{\top} \by^k, x^{*} \rangle + \langle \Delta^k + (A^{\top} \by^k + \bz^k), \bx^{k} \rangle \\
    &= \langle \Delta^k, \bx^k - x^* \rangle + \langle \bz^k, \bx^k \rangle.
\end{aligned}
\]
By \eqref{def:tx} and \eqref{def:tx-prop}, we further derive  
\begin{equation}\label{Th:complexity-acc-pADMM-8}
    \begin{aligned}
        \langle c, \bx^k - x^{*} \rangle 
        &\leq \langle \Delta^k, \bx^k - x^* \rangle + \sigma \langle \bz^k, \Delta^k \rangle \\
        &= \langle \Delta^k, \bx^k - x^* + \sigma (\bz^k - z^{*}) \rangle + \sigma \langle z^*, \Delta^k \rangle \\
        &\leq \| \Delta^k \| \|\bx^k - x^* + \sigma (\bz^k - z^{*})\| + \sigma \|z^*\| \|\Delta^k\|.
    \end{aligned}
\end{equation}
By the definition of \( \cM \) in \eqref{def:M} and the \( \cM \)-nonexpansiveness of \( \hcT \) \cite{bredies2022degenerate}, we have  
\begin{equation}\label{Th:complexity-acc-pADMM-9}
    \begin{aligned}
        \|\bx^k - x^* + \sigma (\bz^k - z^{*})\|^2 
        &= \sigma \|\bw^k - w^{*}\|_{\cM}^2 \\
        &\leq \sigma \|w^0 - w^{*}\|_{\cM}^2 \\
        &= \|x^0 - x^* + \sigma (z^0 - z^{*})\|^2 = R_0^2.
    \end{aligned}
\end{equation}
Therefore, combining \eqref{Th:complexity-acc-pADMM-3}, \eqref{Th:complexity-acc-pADMM-8}, and \eqref{Th:complexity-acc-pADMM-9}, we conclude that  
\[
\begin{aligned}
    \langle c, \bx^k - x^{*} \rangle 
    &\leq \| \Delta^k \| R_0 + \sigma \|z^*\| \|\Delta^k\| \\
    &\leq (\sigma \|z^*\| + R_{0}) \frac{R_0}{\sigma (k+1)}.
\end{aligned}
\]
This completes the proof.
\end{proof}

\section{More numerical results on the DOTmark dataset}\label{more-DOT}
Note that the worst-case computational complexity of reconstructing the transport plan via  Algorithm~\ref{alg:transportplan} is $3M^2$. In practice, we can efficiently parallelize the $(k, j)$ loop to leverage the significant benefits of GPU acceleration, enabling an efficient reconstruction of the transport plan from a solution to the reduced OT model. We provide some additional numerical experiment results in Table \ref{tab:recoveryresult} to better demonstrate the efficiency of the transport plan recovery with the GPU acceleration. The results shown in Table \ref{tab:recoveryresult} demonstrate that our approach outperforms other popular algorithms, even when accounting for the plan recovery time. We have also considered the sparsity in the transport plan recovery, which facilitates handling large-scale problems.

\begin{table*}[ht!]
    \caption{Numerical results of various algorithms on the DOTmark Dataset: Total time as the sum of solve time and transport plan recovery time.}
    \centering
    \resizebox{16cm}{!}{
        \centering
    \begin{tabular}{ccccccccc}
        \multicolumn{3}{c}{} & \multicolumn{1}{c}{} & \multicolumn{1}{c}{}  &
        \multicolumn{1}{c}{} & \multicolumn{1}{c}{} & \multicolumn{1}{c}{} & \multicolumn{1}{c}{}    \\
        \cline{1-9}
        Category    &   Resolution  &   & \textbf{HOT}  & Network Simplex  & Gurobi  & ADMM & Improved Sinkhorn & Sinkhorn \\
        \cline{1-9}
        \multirow{2}{*}{Classic} &   \multirow{1}{*}{$64 \times 64$}     & time(s)  & \textbf{0.67}+0.17 & 2.73 & 2.16+0.17 &  1.77+0.17 & 16.18 &  174.82                      \\
        \cdashline{2-9}
        &   \multirow{1}{*}{$128 \times 128$}     & time(s)  & \textbf{1.58}+0.67 & 36.18 & 29.15+0.67 & 3.53+0.67 & 39.40 &  2632.17                      \\
        \cline{1-9}
        \multirow{2}{*}{Shapes} &   \multirow{1}{*}{$64 \times 64$}     & time(s)  & \textbf{0.64}+0.17 & 1.48 & 1.33+0.17 & 3.92+0.17 & 9.60 & 103.74                   \\
        \cdashline{2-9}
        &   \multirow{1}{*}{$128 \times 128$}     & time(s)  & \textbf{1.68}+0.67 & 20.70 & 22.46+0.67 & 2.32+0.67 &  24.32 & 1616.34    \\
        \cline{1-9}
    \end{tabular}
    }
    \label{tab:recoveryresult}
\end{table*}

The numerical comparison of the HOT algorithm on GPU and CPU is presented in Table \ref{tab:cpu}. The results demonstrate that the GPU's acceleration is substantial, and the acceleration effect becomes more significant as the problem's dimension increases.
    \begin{table}[ht!]
    \centering
    \caption{The comparison of HOT's performance on CPU and GPU.}
    \resizebox{9cm}{!}{
    \begin{tabular}{cccccc}
        \multicolumn{3}{c}{} & \multicolumn{1}{c}{} & \multicolumn{1}{c}{}  & \multicolumn{1}{c}{}  \\
        \hline
        Category    &   Resolution  &   & \textbf{GPU}  & CPU & Ratio ($t_{\text{CPU}} / t_{\text{GPU}}$)  \\
        \hline
        \multirow{3}{*}{Classic} &   
        \multirow{1}{*}{$128 \times 128$}     & time(s)  & \textbf{1.58} & 5.91    &  \multirow{1}{*}{3.74}          \\
        \cdashline{2-6}
        &   \multirow{1}{*}{$256 \times 256$}     & time(s)  & \textbf{12.98} & 108.98   &  \multirow{1}{*}{8.40}          \\
        \cdashline{2-6}
        &   \multirow{1}{*}{$512 \times 512$}     & time(s)  & \textbf{81.02} & 764.22 &  \multirow{1}{*}{9.43}   \\
        \hline
        \multirow{3}{*}{Shapes} 
        &   \multirow{1}{*}{$128 \times 128$}     & time(s)  & \textbf{1.68} & 6.15  &  \multirow{1}{*}{3.66}                      \\
        \cdashline{2-6}
        &   \multirow{1}{*}{$256 \times 256$}     & time(s)  & \textbf{14.87} & 130.81 &  \multirow{1}{*}{8.80}                 \\
        \cdashline{2-6}
        &   \multirow{1}{*}{$512 \times 512$}     & time(s)  & \textbf{87.12} & 812.82 &   \multirow{1}{*}{9.33}     \\
        \hline
    \end{tabular}
    }
    \label{tab:cpu}
\end{table}

To better demonstrate the significance of the Halpern acceleration, we conduct additional numerical testing by changing the stopping criterion of the algorithm from relative KKT residual to absolute KKT residual and setting the maximum number of iterations to 1E6. The results are shown in Table \ref{tab:absKKT} below. 
    
    \begin{table}[ht!]
    \scriptsize
    \centering
    \caption{The numerical results of HOT and ADMM using absolute KKT residual as the stopping criterion (1E-6).}
    \resizebox{9cm}{!}{
    \begin{tabular}{ccccccc}
    \multicolumn{3}{c}{} & \multicolumn{1}{c}{} & \multicolumn{1}{c}{} & \multicolumn{1}{c}{}  & \multicolumn{1}{c}{} \\
    \hline
    Category & Resolution &  & HOT & ADMM   \\
    \hline 
    \multirow{11}{*}{Classic} & \multirow{4}{*}{$64 \times 64$} & time(s) & 6.39 & 406.94 \\ 
                              &                                 & gap & \textbf{6.93E-10} & 1.58E-6 \\
                              &                                 & feaserr & \textbf{3.12E-13} & 4.51E-10 \\
                              &                                 & iter & 17680 & 1000000  \\
    \cdashline{2-7} 
                              & \multirow{4}{*}{$128 \times 128$} & time(s) & 76.36 & 1093.10 \\ 
                              &                                 & gap & \textbf{8.10E-10} & 4.69E-6 \\
                              &                                 & feaserr & \textbf{2.82E-14} & 6.11E-10 \\
                              &                                 & iter & 57300 & 1000000 \\
    \cdashline{2-7} 
                              & \multirow{4}{*}{$256 \times 256$} & time(s) & 3774.04 & 9135.66 \\ 
                              &                                 & feaserr & \textbf{3.96E-15} & 9.31E-11\\
                              &                                 & iter & 338240 &  1000000\\
    \hline
    \multirow{11}{*}{Shapes}  & \multirow{4}{*}{$64 \times 64$} & time(s) & 6.34 & 404.29 \\ 
                              &                                 & gap & \textbf{2.29E-10} & 6.48E-7\\
                              &                                 & feaserr & \textbf{3.62E-13} & 8.78E-10\\
                              &                                 & iter & 17660 & 991730\\
    \cdashline{2-7} 
                              & \multirow{4}{*}{$128 \times 128$} & time(s) & 54.86 & 1087.06\\ 
                              &                                 & gap & \textbf{2.47E-9} & 2.43E-6 \\
                              &                                 & feaserr & \textbf{2.33E-14} & 1.01E-9\\
                              &                                 & iter & 41190 & 1000000\\
    \cdashline{2-7} 
                              & \multirow{4}{*}{$256 \times 256$} & time(s) & 2519.99 & 9101.29\\ 
                              &                                 & feaserr & \textbf{3.24E-15} & 1.10E-10\\
                              &                                 & iter & 226020 & 1000000 \\
    \hline
    \end{tabular}
    }
    \label{tab:absKKT}
    \end{table}

\section{Compare with W2NeuralDual} \label{comp_W2NeuralDual}
In this section, we compare HOT with the W2NeuralDual method implemented in OTT-JAX \cite{cuturi2022optimal}.
We first sample some data using a built-in OTT-JAX random dataset generator in two different settings, visualized in Fig. \ref{fig:ott_data}.  For the Circle Gaussian and the CheckerBoard dataset, we generate 128 data points for both the source and target measures, respectively. Then we compare the performance of the HOT algorithm and the neural network dual method with default parameters in the OTT-JAX library\footnote{\url{https://ott-jax.readthedocs.io/en/latest/neural/_autosummary/ott.neural.methods.neuraldual.W2NeuralDual.html}}.
    Specifically, we train the network using a batch size of 2048 and utilize the ``W2NeuralDual'' API without modifying any parameters (e.g., 20000 training epochs, Adam optimizers). Note that the ICNN potential, as the default parameter, is not suitable for training the CheckerBoard dataset (see Fig. \ref{fig:ICNNCheckBoard}). We manually modify it to MLP potential as suggested by OTT-JAX documentation\footnote{\url{https://ott-jax.readthedocs.io/en/latest/tutorials/neural/000_neural_dual.html}}.
    Here, we also use the 2D non-uniform bins constructed from the union of the supports in the source and target measure samples.
     We report the numerical results in Table \ref{tab:gaussianres}. From the results, we can see that, even though the W2NeuralDual is computationally efficient in the inference stage and is memory efficient, the training phase is computationally expensive. Moreover, the quality of the solution obtained by the HOT algorithm can be much better than the ones obtained from the W2NeuralDual. We further visualize the transport plans obtained by the HOT algorithm and the W2NeuralDual in Fig. \ref{fig: comp-JAX-results}.
    
    \begin{figure}[h!]
        \centering
        \begin{subfigure}[b]{0.48\textwidth}
         \centering
         \includegraphics[width=\textwidth]{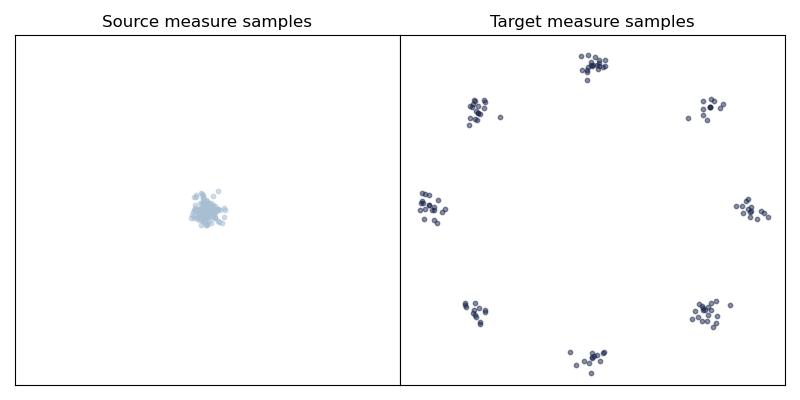}
         \caption{Circle Gaussian.}
     \end{subfigure}
         \begin{subfigure}[b]{0.48\textwidth}
         \centering
         \includegraphics[width=\textwidth]{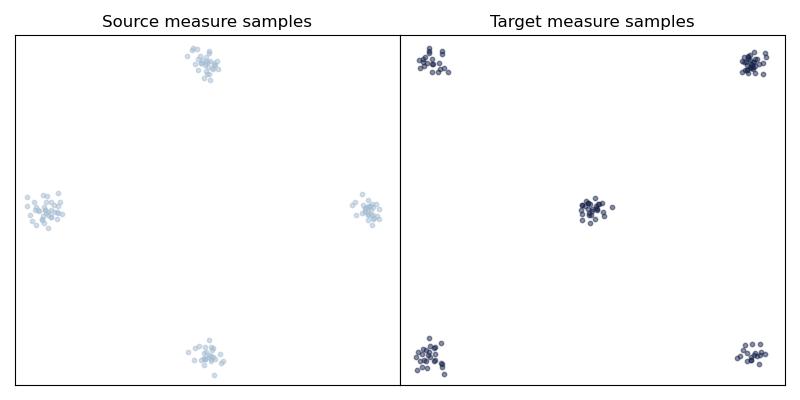}
         \caption{CheckerBoard.}
     \end{subfigure}
     \caption{A visualization of the sample dataset.}
     \label{fig:ott_data}
    \end{figure}

    \begin{figure}[h!]
        \centering
        \includegraphics[width=0.7\linewidth]{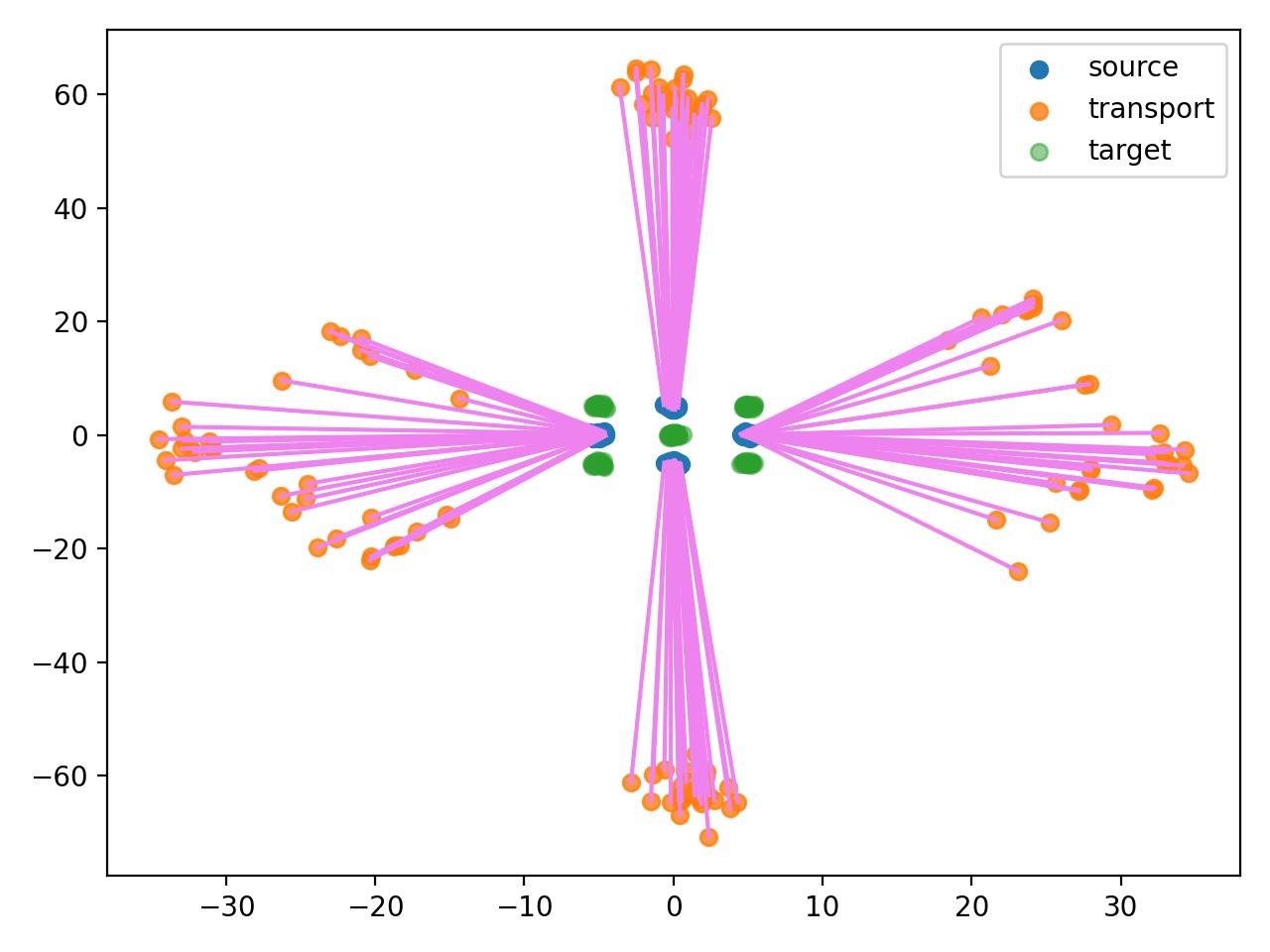}
        \caption{W2NeuralDual on CheckerBoard Using ICNN Potential with default parameter settings.}
        \label{fig:ICNNCheckBoard}
    \end{figure}

\begin{figure}[htbp]
    \centering
    \fbox{
    \begin{subfigure}[t]{0.2\textwidth}
        \centering
        \begin{subfigure}[t]{\textwidth}
            \centering
            \includegraphics[width=\textwidth]{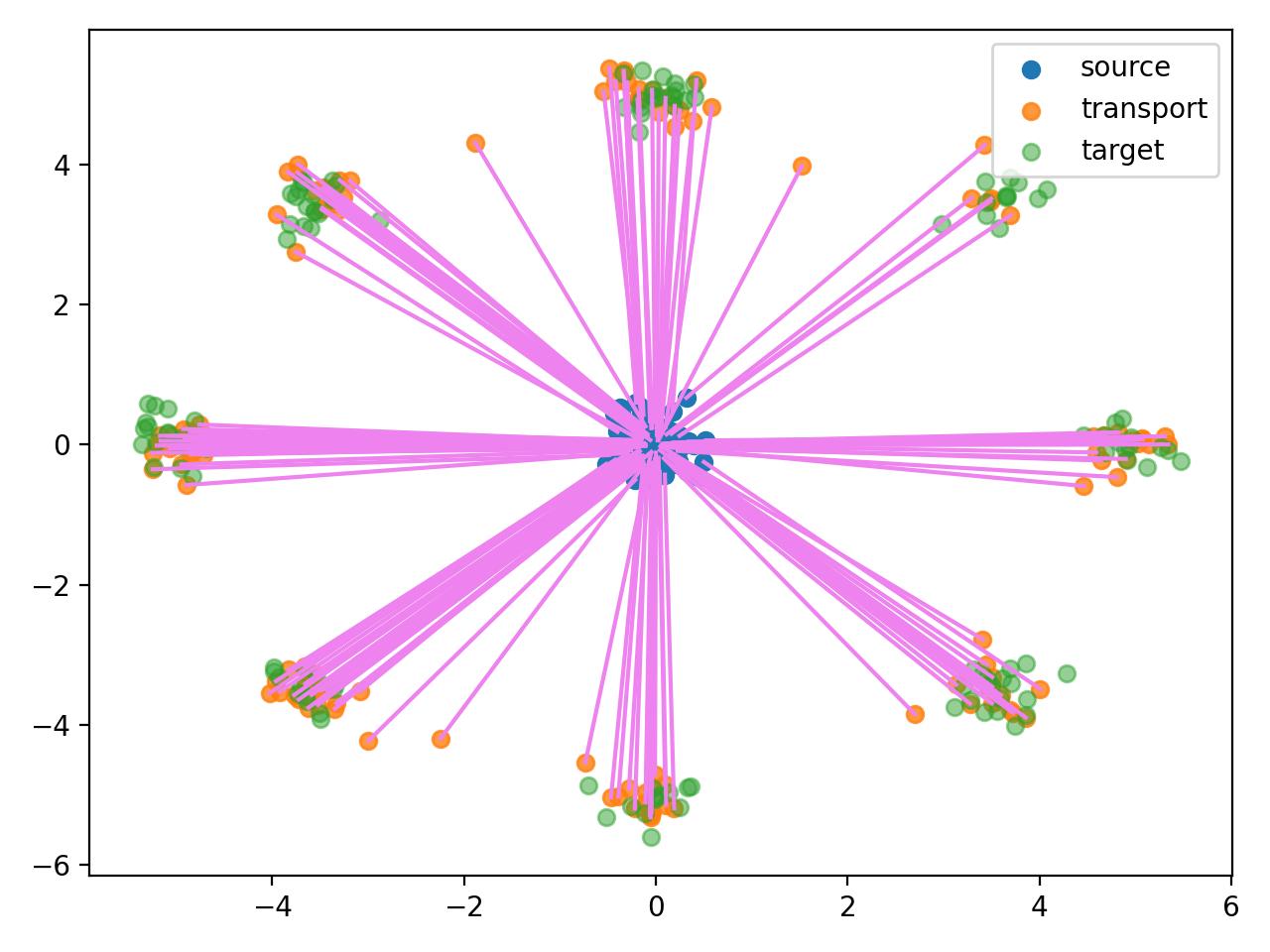}
            \caption*{W2NeuralDual}
        \end{subfigure}\\
        \begin{subfigure}[t]{\textwidth}
            \centering
            \includegraphics[width=\textwidth]{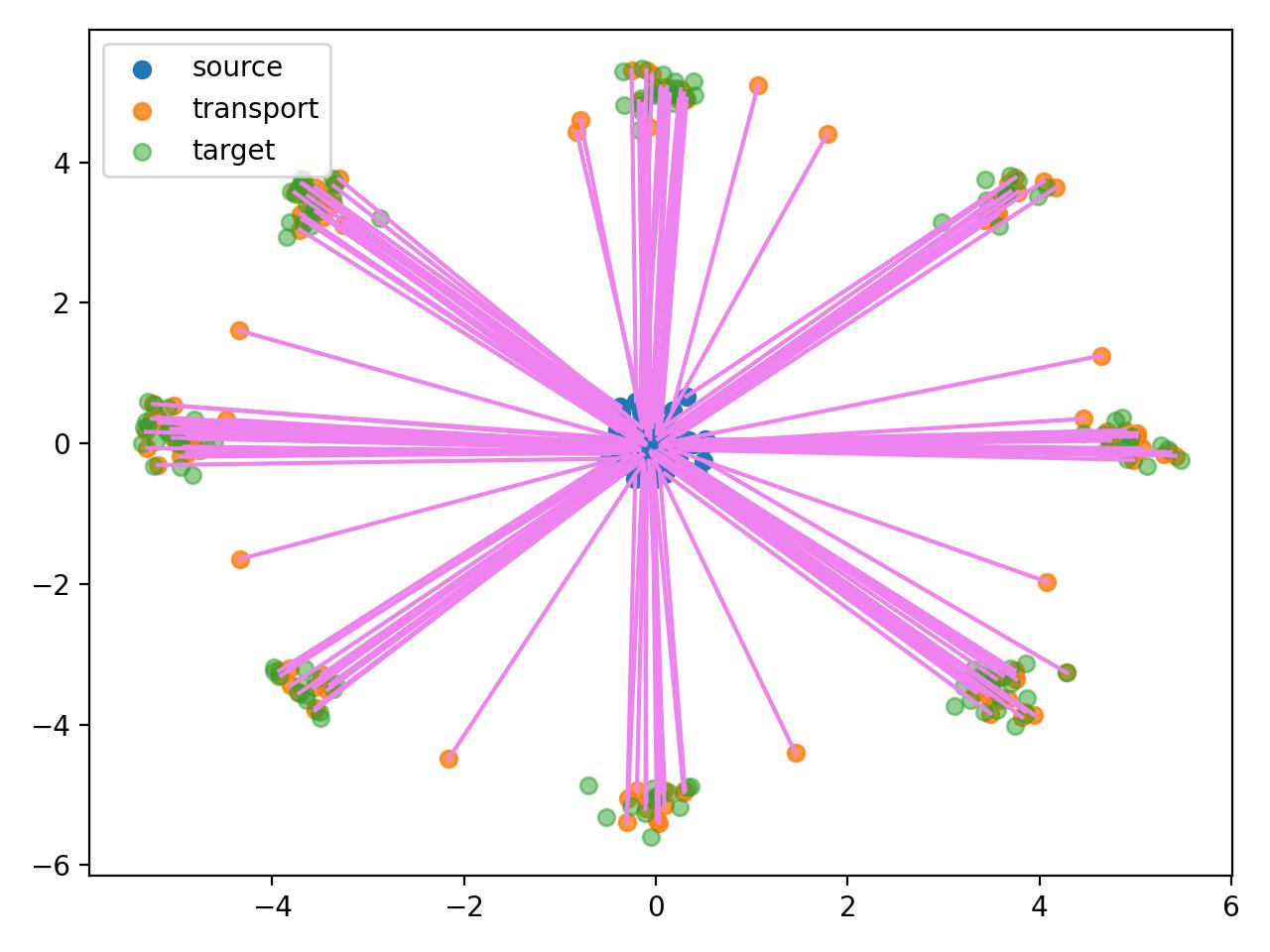}
            \caption*{HOT tol=1E-6}
        \end{subfigure}\\
        \begin{subfigure}[t]{\textwidth}
            \centering
            \includegraphics[width=\textwidth]{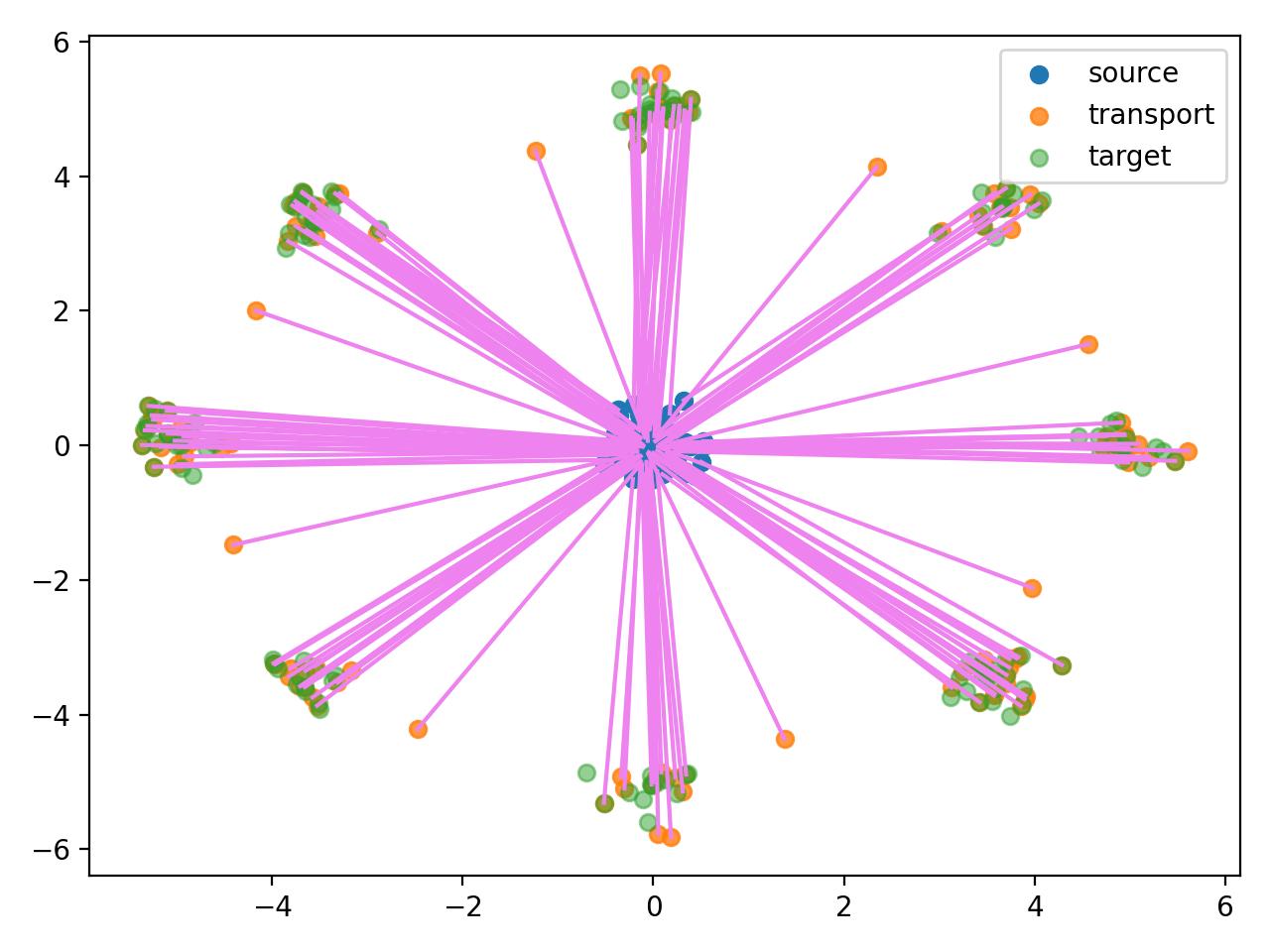}
            \caption*{HOT tol=1E-7}
        \end{subfigure}
        \caption{Comparison on Circle Gaussian dataset.}
    \end{subfigure}
    }
    \fbox{
    \begin{subfigure}[t]{0.2\textwidth}
        \centering
        \begin{subfigure}[t]{\textwidth}
            \centering\includegraphics[width=\textwidth]{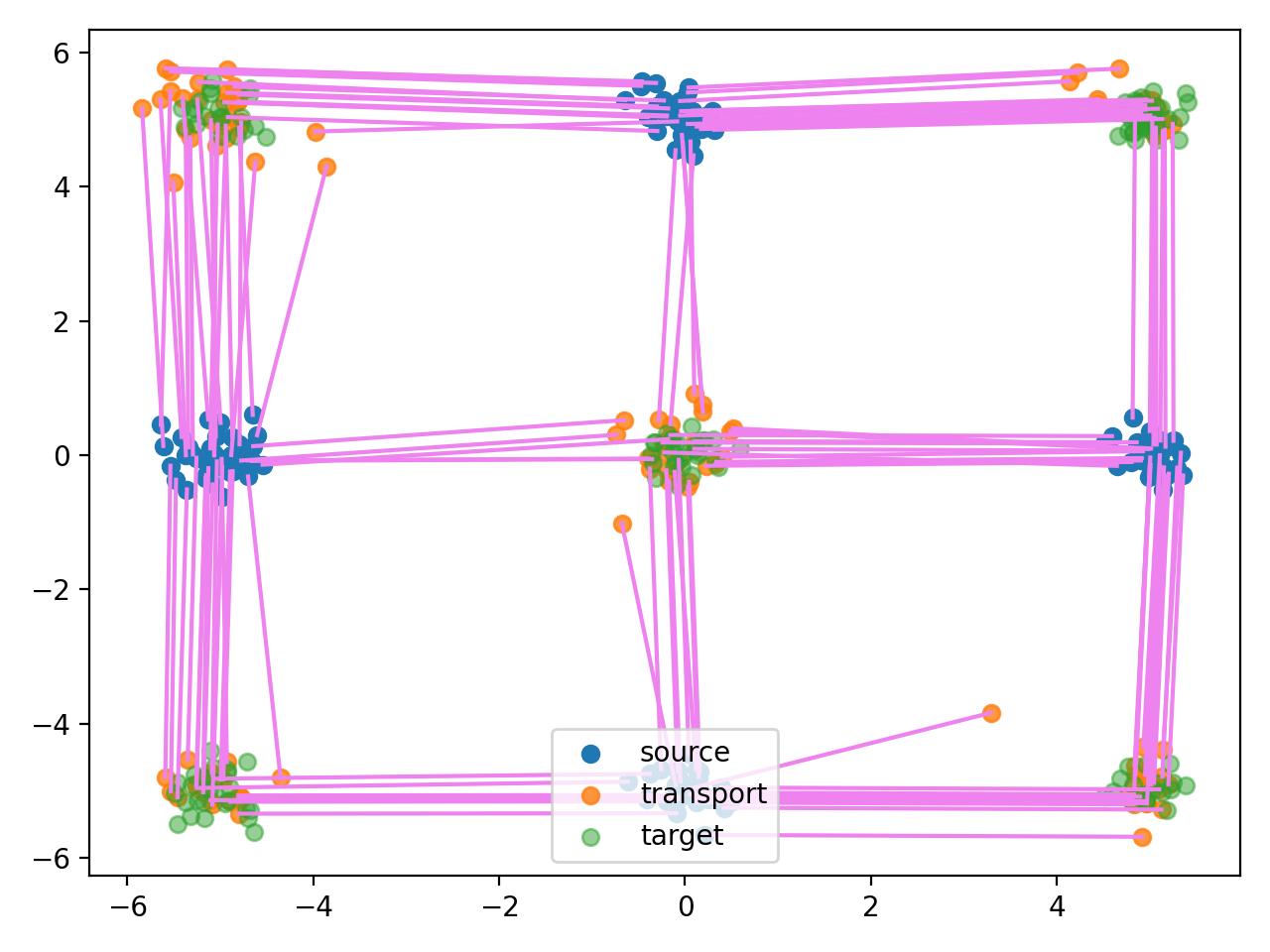}
            \caption*{W2NeuralDual}
        \end{subfigure}\\
        \begin{subfigure}[t]{\textwidth}
            \centering
            \includegraphics[width=\textwidth]{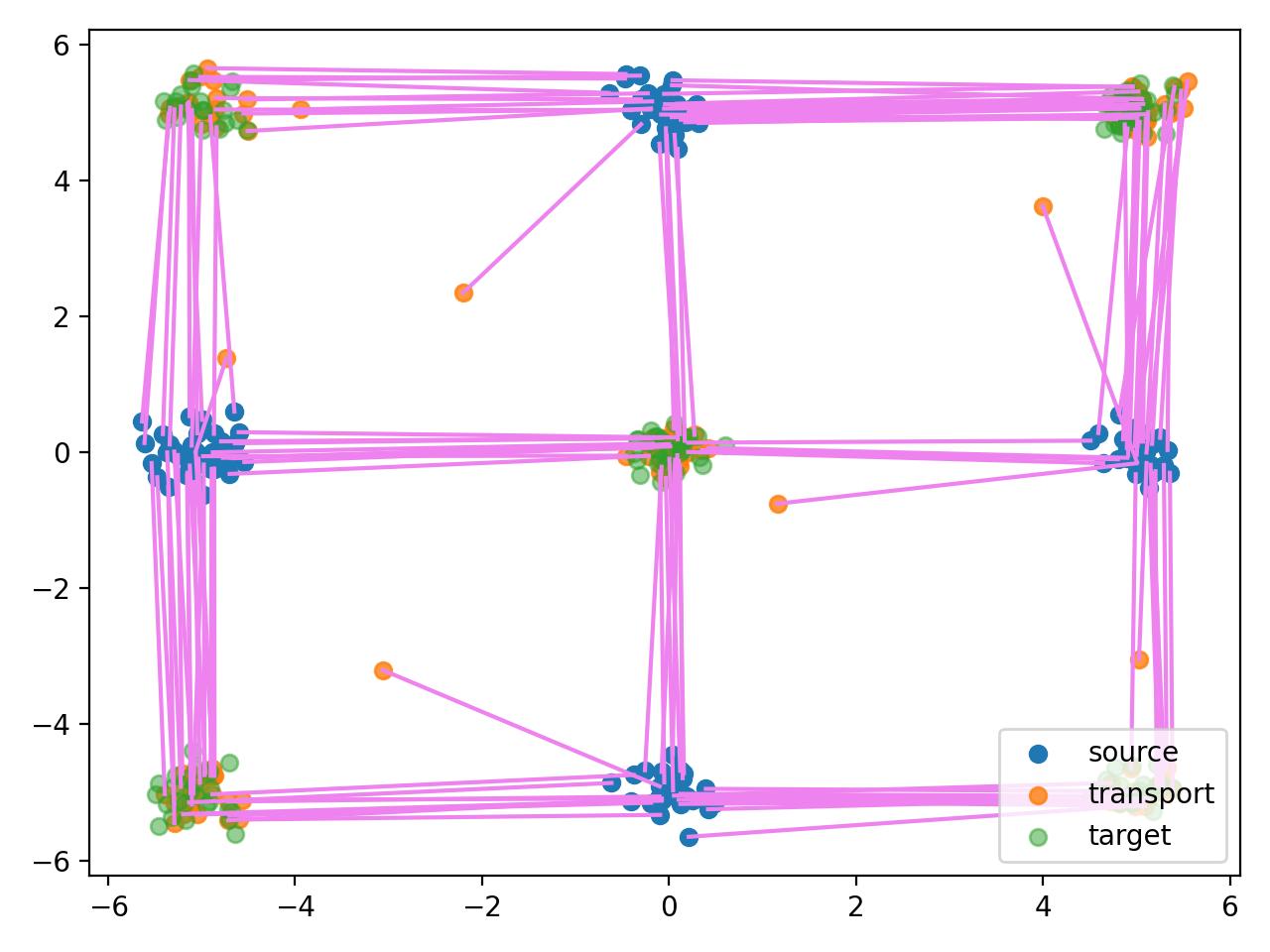}
            \caption*{HOT tol=1E-6}
        \end{subfigure}\\
        \begin{subfigure}[t]{\textwidth}
            \centering
            \includegraphics[width=\textwidth]{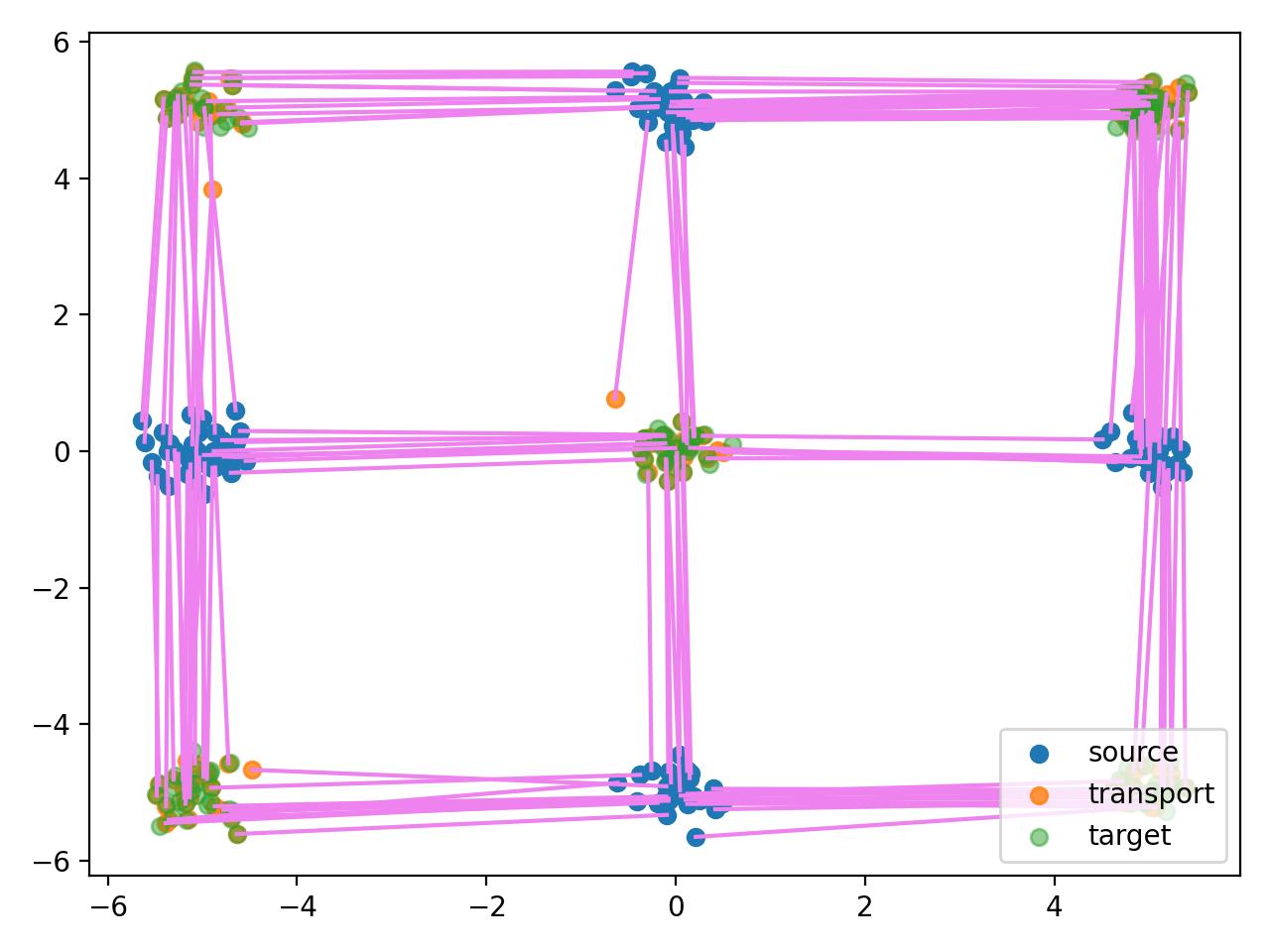}
            \caption*{HOT tol=1E-7}
        \end{subfigure}
        \caption{Comparison on CheckerBoard dataset.}
    \end{subfigure}
    }
    \caption{A visualization of the transport plans on different datasets.}
    \label{fig: comp-JAX-results}
\end{figure}

 \begin{table}[h!]
        \centering
        \resizebox{9cm}{!}{
        \begin{tabular}{cccccc}
        \hline
        Dataset& & \multicolumn{2}{c}{HOT} & \multicolumn{2}{c}{W2NeuralDual}  \\
              &  & tol=1E-6 & tol=1E-7 & training & inference \\ 
        \hline
        \multirow{4}{*}{\shortstack{Circle Gaussian}} & objective & 22.3634 & 22.3601 & $\backslash$ & 22.2784 \\
        & gap       & 1.54E-4 & 1.32E-5 &  $\backslash$  &  3.49E-3  \\
                                         & time(s) & 51.73 & 262.78 & 4466.61 & 5.21 \\ 
                                         & memory(GB) & 2.72 & 2.72 & 0.65 & $\backslash$ \\
        \hline 
        \multirow{4}{*}{\shortstack{CheckerBoard}} & objective & 21.5437 & 21.5348 & $\backslash$ & 23.6818 \\
        & gap  &  4.09E-4 & 1.46E-5 &   $\backslash$ &  9.53E-2  \\
                                 & time(s) & 62.51 & 272.29 & 3029.24 & 3.58 \\ 
                                 & memory(GB) & 2.72 & 2.72 & 0.65 & $\backslash$ \\
        \hline
        \end{tabular} 
        }
        \caption{Numerical results of HOT and W2NeuralDual on the sample dataset.}        \label{tab:gaussianres}     
    \end{table}

\section{Domain adaptation}\label{domain-adaptation}
Optimal transport has a wide range of applications in domain adaptation and data alignment.  
Inspired by \cite{courty2016optimal}, we apply the HOT algorithm for domain adaptation in unsupervised classification. We first sample test data from the target domain and perform optimal transport between the source domain and the sampled test data. We then train a classifier on the transported data, which retains the source domain labels, and evaluate it on the target domain. Again, we apply the non-uniform 2D grids construction method here.
        \begin{figure}[ht!]
            \centering
            \begin{subfigure}[b]{0.24\textwidth}
            \centering
            \includegraphics[width=\textwidth]{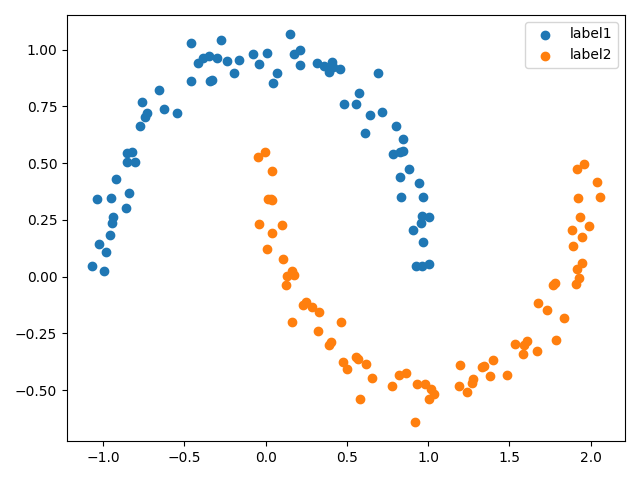}
            \end{subfigure}
            \begin{subfigure}[b]{0.24\textwidth}
            \centering
            \includegraphics[width=\textwidth]{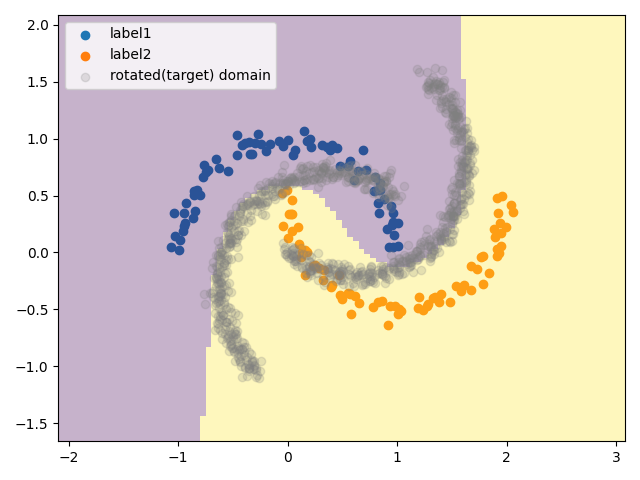}  
            \end{subfigure}
            \begin{subfigure}[b]{0.24\textwidth}
            \centering
            \includegraphics[width=\textwidth]{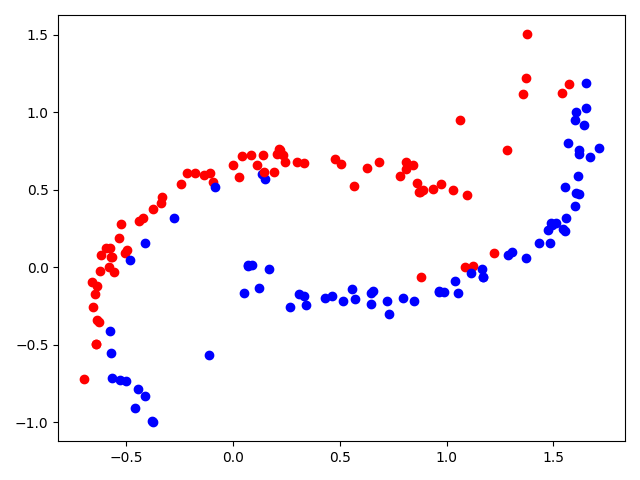}  
            \end{subfigure}
            \begin{subfigure}[b]{0.24\textwidth}
            \centering
            \includegraphics[width=\textwidth]{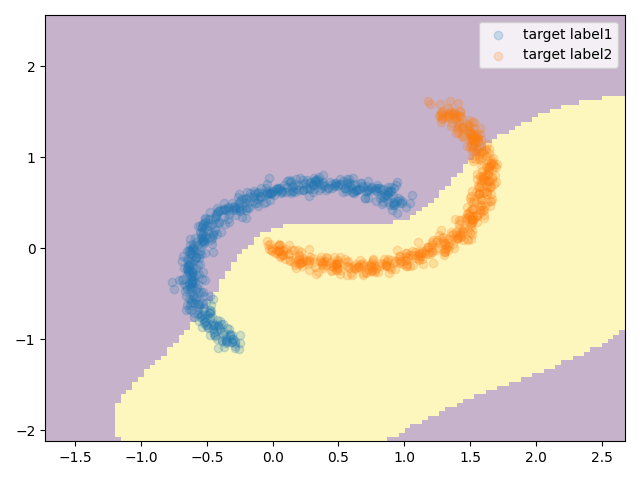}  
            \end{subfigure}
            
            \caption{Optimal transport based domain adaptation for unsupervised classification.}
            \label{fig:otda}
        \end{figure}

From Table \ref{tab:otdaresult}, we can conclude that optimal transport-based domain adaptation significantly increases classification accuracy.

        \begin{table}[h!]
            \centering
            \caption{Influence of optimal transport based domain adaptation on unsupervised classification accuracy.}
            \begin{tabular}{ccccccc}
            \hline
            Rotation Angle & 10$^\circ$ & 20$^\circ$ & 40$^\circ$ & 50$^\circ$ & 70$^\circ$ & 90$^\circ$ \\
            \hline
            SVM     & 100\% & 100\% & 76.5\% & 40.4\% & 26.2\% & 18.7\%                   \\
            \hline
            OTDA-SVM & 100\% & 100\% & 92.6\% & 82.8\% & 68.4\% & 59.0\%                      \\
            \hline
            \end{tabular}
            \label{tab:otdaresult}
        \end{table}

\section{More examples of color transfer} \label{more_application}
We present additional examples of color transfer by the HOT algorithm, as shown in Fig. \ref{fig:color_transfer}.

   \begin{figure*}[!ht]
    \centering
    \includegraphics[width=0.80\textwidth]{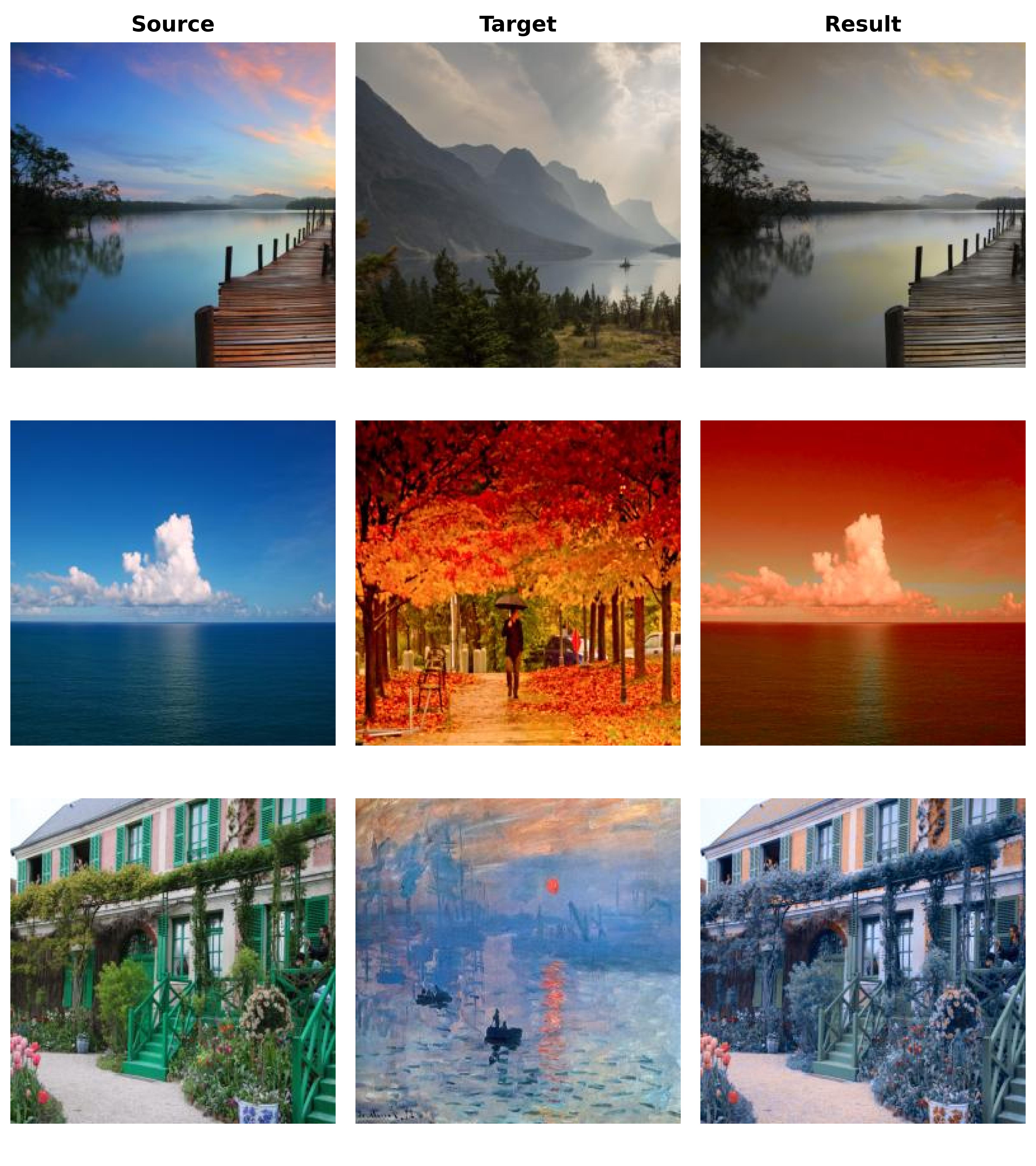}
    \caption{More examples of color transfer.}
    \label{fig:color_transfer}
    \end{figure*}

\putbib
\end{bibunit}

\end{document}